\documentclass{amsart}
\usepackage{graphicx}
\usepackage{amssymb}
\usepackage{amsmath}
\usepackage{amsthm}
\usepackage{amsaddr}
\usepackage{enumerate}
\usepackage{bbm}
\usepackage{caption}
\usepackage{subcaption}
\usepackage{calligra}
\usepackage{MnSymbol}
\usepackage{bbm}
\usepackage{yfonts}
\usepackage{upgreek}
\usepackage{amssymb} 
\usepackage{adforn}
\usepackage{pgfornament}
\usepackage{enumerate}
\usepackage{verbatim}
\usepackage{collectbox}
\usepackage{xcolor}
\usepackage{MnSymbol}

\DeclareMathAlphabet{\mathcalligra}{T1}{calligra}{m}{n}

\newtheorem{thm}{Theorem}[section]
\newtheorem{cor}[thm]{Corollary}

\newtheorem{lemma}[thm]{Lemma}
\newtheorem{prop}[thm]{Proposition}
\theoremstyle{definition}
\newtheorem{definition}[thm]{Definition}
\newtheorem{remark}[thm]{Remark}

\newtheorem*{defn*}{Definition}
\newtheorem*{rems*}{Remarks}
\newtheorem*{rem*}{Remark}

\newtheorem{corollary}[thm]{Corollary}

\usepackage[pagewise]{lineno}

\numberwithin{equation}{section}

\begin{document}

\title[The Geometry of the Centre Symmetry Set of a~Planar Curve]{The Geometry of the Centre Symmetry Set of a~Planar Curve}

\author{Dominika Miller$^{1}$, Micha\l{} Zwierzy\'nski$^{2,\pentagram}$}
\email{$^{1}$dominika.monka.stud@pw.edu.pl}
\email[Corresponding author]{$^{2}$michal.zwierzynski@pw.edu.pl}

\address{Warsaw University of Technology\\
Faculty of Mathematics and Information Science\\
ul. Koszykowa 75\\
00-662 Warsaw, Poland}

\thanks{$^{\pentagram}$Corresponding author}



\subjclass[2010]{53A04, 53A15, 58K05}

\keywords{Centre Symmetry Set, Wigner caustic, singularities, generic}


\newcommand{\Eq}{{\text{E}}}
\newcommand{\CSS}{\mathrm{CSS}}
\newcommand{\minm}{\mathbbm{m}}
\newcommand{\maxm}{\mathbb{M}}
\newcommand{\overarc}[2]{\underline{#1\frown #2}}
\newcommand{\p}{p}
\newcommand{\SC}{\mathrm{SC}}


\begin{abstract}
The Centre Symmetry Set of a planar curve $M$ is the envelope of affine chords of $M$, i.e. the lines joining points on $M$ with parallel tangent lines. In this paper we study global geometrical properties of this set including the number of singularities and the number of asymptotes. 
\end{abstract}

\maketitle


\section{Introduction}

Let $M$ be a smooth curve in an affine space $\mathbb{R}^2$. Consider an affine chord of $M$, i.e. an infinite line joining pairs of points of $M$ at which the tangent lines to $M$ are parallel. Then, the Centre Symmetry Set of $M$ ($\CSS(M)$) is the envelope of all affine chords of $M$. Note that for a convex $M$ centrally symmetric about a point $p$ the Centre Symmetry Set degenerates to $p$. Therefore, the $\CSS$ measures in some way the central symmetric property of a curve. The origin of this investigation is the paper \cite{Janeczko} by S. Janeczko. He described the Centre Symmetry Set by the bifurcation set of a certain family of ratios. Then, P.J.~Giblin and P.A.~Holtom in \cite{GiblinHoltom} redefined the $\CSS$ in the way stated above. 
The various kinds of symmetries are a vital part of the study of manifolds in Euclidean spaces and in their geometry applications. The medial axis transform is one of the most popular examples of such applications (see \cite{medial2, medial1}).

The Centre Symmetry Set of a curve can be viewed also as the set of all singular points of a family of affine $\lambda$-equidistants of the curve (see for instance \cite{DJRR, DomitrzMR, DomitrzRiosRuas, ZD_Wigner, GiblinZ2, Zwierz1, Zwierz3} and the references given there. See also \cite{CraizerMartini, CraizerPPOS} for the Centre Symmetry Set in a non-smooth, polygonal case). Generically, the Centre Symmetry Set of a planar curve can admit at most cusp singularities. It was studied by many authors, also in higher dimensions, and in Minkowski spaces. For instance, authors in \cite{GiblinZ1, GiblinZ2, GiblinZ3, ReeveZ} analyse the local properties of the $\CSS$ based on the theory of Lagrange and Legendre singularities. In \cite{GiblinReeve, GiblinReeve2} the authors study the Centre Symmetry Set of families of plane curves and of families of surfaces. The $\CSS$ is a part of the Global Centre Symmetry Set studied in \cite{DomitrzRios}. To see art created from the Centre Symmetry Set (and the Wigner caustic) see \cite{Art}. In the articles cited, the authors obtained properties of the $\CSS$ that are inherently local. One of the few global properties is the result in \cite{GiblinHoltom} that the number of cusps of $\CSS$ of a generic oval is odd and not smaller than $3.$ The results in \cite{Craizer0, DomitrzRios} show that the number of singularities of the Centre Symmetry Set is greater than the number of singularities of the Wigner caustic of a planar oval (in a generic situation). Moreover, it is known that the Wigner caustic is contained in a region bounded by the Centre Symmetry Set of an oval (\cite{Craizer0}). In this article we will prove many global geometrical properties of the Centre Symmetry Set of a planar curve, not necessarily an oval. These theorems will concern the number of the so-called smooth branches of the $\CSS$, the number of singular points, and the number of asymptotes of this set. Before we come to these theorems, in the next section we will precisely state the conditions for specific types of singularities and directly define the term of generic curves, used throughout the article (see Theorem \ref{ThmGenericCSS}).

In Section \ref{decomp} we adapt the decomposition of a curve into parallel arcs from \cite{ZD_Wigner}. This decomposition is necessary to prove some of the global properties of the Centre Symmetry Set of a planar curve in the last section.

\section{The Geometry of the Centre Symmetry Set}\label{geomCSS}

Let $C^{\infty}(X,Y)$ denote the set of smooth $(C^\infty)$ maps between smooth manifolds $X$ and $Y$ with the smooth Whitney topology. A set in a topological space is \textit{generic} if it is a countable intersection of open and dense sets. A \textit{Baire space} is a topological space in which any generic set is dense. It is well known that $C^{\infty}(X,Y)$ is a Baire space (e.g. see \cite{GG-Book}).
Throughout the paper we will denote by $M$ the image of a \textit{smooth curve} (i.e. an element in $C^{\infty}(I,\mathbb{R}^2)$, where $I\subset\mathbb{R}$ is an open interval) on the affine plane $\mathbb{R}^2$ or the smooth curve itself if this does not lead to ambiguity. If a smooth map is from $S^1$ to $\mathbb{R}^2,$ we will call its image \textit{closed}. A point of a curve is called \textit{regular} if its velocity vector (the derivative) does not vanish. Otherwise -- it is called \textit{singular}. A curve without a singular point will be called \textit{regular}. Otherwise -- \textit{singular}. A regular curve is \textit{convex} if its curvature does not vanish at any point. The \textit{rotation number} of a regular curve is the rotation number of its velocity vector field. An \textit{$n$-rosette} is a closed convex curve with $n$ as its rotation number. A point $p$ of a smooth regular curve is an \textit{inflexion point} if its curvature changes sign at $p$, and is an \textit{undulation point} if its curvature vanishes at $p$ but does not change sign at $p$. If a regular curve $M$ is parameterized by $f$, then an inflexion point $f(s_0)$ is \textit{non-degenerate} if $\det\left(f'(s_0),f'''(s_0)\right)\neq 0$. A singular point is a \textit{cusp} if its locally diffeomorphic (in the source and in the target) to the map $t\mapsto(t^2,t^3)$ at $t=0$. It is well known (e.g. see Theorem B.9.1 in \cite{UYBook}) that a map $f$ is a cusp at $s_0$ if and only if $f'(s_0)=0$ and $\det\big(f''(s_0),f'''(s_0)\big)\neq 0.$ 
We will denote by $\kappa_f$ or $\kappa_M$ the signed curvature of a curve $M$ parameterized by $f.$

\begin{definition}
A pair of points $a,b \in M$ $(a \neq b)$ is called a \textit{parallel pair} if the lines tangent to $M$ at $a$ and $b$ are parallel.
\end{definition}

\begin{definition}
A \textit{chord} passing through a pair $a,b \in M,$ is the line:
\begin{equation}
    \ell(a,b) = \{ \lambda a + (1-\lambda)b  \ | \ \lambda \in \mathbb{R} \}.
\end{equation}
If $a,b$ is a parallel pair of $M,$ then $\ell(a,b)$ is called an \textit{affine chord}.
\end{definition}

\begin{definition}
The \textit{Centre Symmetry Set} of $M,$ denoted by $\CSS(M),$ is the envelope of all affine chords of $M.$
\end{definition}

\begin{definition}
Let $\lambda\in\mathbb{R}$. An affine $\lambda$ equidistant of $M$ is the following set
$$\Eq_{\lambda}(M):=\left\{\lambda a+(1-\lambda)b\, \big|\, a,b\text{ is a parallel pair of }M\right\}.$$
The set $\Eq_{0.5}(M)$ is also known as the \textit{Wigner caustic} of $M$. 
\end{definition}

The Centre Symmetry Set of $M$ can be also viewed as the set of all singular points in the family of affine $\lambda$-equidistants for $\lambda\in\mathbb{R}$ (\cite{GiblinZ2,GiblinZ3}). There are many articles related with the geometry of singularities of affine equidistants -- see \cite{DJRR, DomitrzRiosRuas, Zwierz1, Zwierz2, Zwierz3} and their references. In particular, the geometry of the Wigner caustic is strictly related with isoperimetric problems (\cite{Zhang1}--\cite{Zwierz3}) and with the construction of improper affine spheres (\cite{CraizerDR1, CraizerDR2, Craizer}).

\begin{definition}
Let $f$ be a regular parameterization of a curve $M.$ We say that $M$ is \textit{parameterized at opposite directions} at points $f(s_1)$ and $f(s_2)$ if $f'(s_1)=\alpha f'(s_2),$ where $\alpha<0$.
\end{definition}

\begin{definition}
Let $M$ be a regular closed curve. 
A pair of points $a, b \in M$ is called \textit{standard} if $a,$ $b$ is a parallel pair, such that $M$ is parameterized at $a$ and $b$ at opposite directions and $\kappa_M(b) \neq 0.$
\end{definition}

\begin{definition}
Let $M$ be a regular closed curve.
A line that is tangent to $M$ at two distinct points is called a \textit{double tangent}.
\end{definition}

In this paper a singular point of the Centre Symmetry Set will be understood as a singular point of its so called \textit{natural parameterization}. This parameterization will be given locally and we will describe it in Subsection \ref{sectionLocParam}.

\subsection{Local Parameterization of $\CSS(M)$}\label{sectionLocParam}

Let $a, b \in M$ be a standard pair of points.
Let $f(s)$ and $g(t)$ be the local arc length parameterizations of $M$ at $a$ and $b,$ respectively.

Since $a$ and $b$ are a parallel pair and $f$ and $g$ are parameterized at opposite directions, we have that $f'(s)=-g'(t),$ where $'$ signifies the derivative with respect to the corresponding parameter $t$ or $s.$ By the implicit function theorem one can show that there exists a smooth function $t(s)$ such that $f'(s)=-g'(t(s))$ and 
\begin{equation}\label{t_prim}
     t'(s) = \dfrac{\kappa_f(s)}{\kappa_g(t(s))}.
\end{equation}
Using this and the fact that the Centre Symmetry Set is an envelope of the following family of lines:
\begin{equation}\label{l}
     \ell_s(u) = f(s) + u(g(t(s)) - f(s)),
\end{equation}
one can get that double tangents are part of the $\CSS$ and the other part can be locally parameterized as follows:
\begin{align}\label{param_l_ostateczna}
     \gamma_{\CSS}(s):=\ell_s\left(\frac{\kappa_g(t(s))}{\kappa_f(s) + \kappa_g(t(s))}\right) 
     = \dfrac{\kappa_f(s) \cdot f(s) + \kappa_g(t(s)) \cdot g(t(s))}{\kappa_f(s) + \kappa_g(t(s))}
\end{align}
whenever $\kappa_f(s)+\kappa_g(t(s))\neq 0.$ 

\begin{remark}\label{rem:asymptote}
In the case when $\kappa_f(s)+\kappa_g(t(s))=0,$ the affine chord through $f(s)$ and $g(t(s))$ is an asymptote of the $\CSS.$
\end{remark}

\begin{remark}
In the neighbourhood of every inflexion point $p_i$ there is a sequence of parallel pairs, in which both points converge to the inflexion point $p_i.$ In other words, for each member $(a_k, b_k)$ of the sequence, $a_k \to p_i$ and $b_k \to p_i.$ This observation leads to the fact that every ordinary inflexion point is the beginning of a component of $\CSS(M)$ (\cite{GiblinHoltom}). This component of $\CSS(M)$ we will call \textit{a $\CSS$ on shell} (similarly to the notion of the Wigner caustic on shell -- see \cite{DomitrzMR, ZD_Wigner}).
\end{remark}

Using the parameterization \eqref{param_l_ostateczna}, by direct calculations one can obtain the following lemma.

\begin{lemma}\label{CuspsofCSS}
Let $a, b \in M$ be a standard pair of points.
The set $\CSS(M)$ has a singular point at $\frac{\kappa_M(a) \cdot a + \kappa_M(b) \cdot b}{\kappa_M(a) + \kappa_M(b)}$ if and only if $\left( \frac{\kappa_M(a)}{\kappa_M(b)} \right)' = 0$ (which is equivalent to $\kappa_M'(a)\kappa_M^2(b)=\kappa_M'(b)\kappa_M^2(a)$), where $'$ signifies the derivative with respect to the arc parameter of the curve. The singular point $p$ of $\CSS(M)$ is a cusp if and only if $\det\left(f(a)-f(b),f'(b)\right)\neq 0$ and  $\left( \frac{\kappa_M(a)}{\kappa_M(b)} \right)'' \neq 0$ (which is equivalent to $\kappa_M''(a)\kappa_M^3(b)\neq\kappa_M''(b)\kappa_M^3(a)$).
\end{lemma}

The above lemma and Corollary \ref{InflPtCssRemark} was also obtained in \cite{GiblinHoltom} using a slightly different approach.

Directly by the local parameterization of the Centre Symmetry Set we obtain the following formula for the curvature function of the $\CSS.$

\begin{prop}\label{PropCurvCss}
Let $a, b \in M$ be a standard pair of points and let \linebreak $q = \frac{\kappa_M(a) \cdot a + \kappa_M(b) \cdot b}{\kappa_M(a) + \kappa_M(b)}$ be a regular point of $\CSS(M).$
Then the curvature of $\CSS(M)$ at $q$ is expressed by the formula
\begin{equation}\label{curvature_CSS}
    \kappa_{\CSS(M)}(q) = \mathrm{sgn}(\kappa_M(b)) \cdot
    \dfrac{\big( \kappa_M(a) + \kappa_M(b)\big)^3}
    {\left| \kappa^2_M(b)\kappa'_M(a) - \kappa^2_M(a)\kappa'_M(b)\right|} \cdot 
    \dfrac{\det(a-b, \mathbbm{t}(a))}{|a-b|^3},
\end{equation}
where $\mathbbm{t}(a)$ is a tangent unit vector field compatible with the orientation of $M$ at $a$ and $'$ signifies the derivative with respect to the arc parameter of a curve.
\end{prop}

\begin{cor}\label{InflPtCssRemark}
Let $a, b \in M$ be a standard pair of points. Any curve contained in $\CSS(M)$ has curvature equal to $0$ at point $p$ if and only if $p$ lies on a double tangent.
\end{cor}

When two asymptotes of a curve are covering each other, which in fact means that they are the same line, this line is called a \textit{double asymptote}.

Before we describe known local properties of the Centre Symmetry Set, we will define a generic set $\mathcal{G}$ in $C^{\infty}(S^1,\mathbb{R}^2)$.
Whenever in the paper we say that the curve $f$ is generic we mean that $f$ belongs to the set described in the following theorem.

\begin{thm}\label{ThmGenericCSS}
Let $\mathcal{G}$ be the subset of $C^{\infty}(S^1,\mathbb{R}^2)$ such that any $f\in\mathcal{G}$ satisfies the following properties:
\begin{enumerate}[(i)]
    \item The curve $f$ is a regular curve with only transversal self-crossings, with only non-degenerate inflexion points and with no undulation points.
    \item If $f(s_1),f(s_2)$ is a parallel pair, then at least one of $f(s_1)$, $f(s_2)$ is not an inflexion point.
    \item Let $f(s_1),f(s_2)$ be a parallel pair. If $\det\big(f(s_1)-f(s_2),f'(s_1)\big)=0$, then $\kappa_f(s_1)\kappa_f(s_2)\neq 0$.
    \item Let $f(s_1), f(s_2)$ be a parallel pair. If $\kappa_f'(s_1)\kappa_g^2(s_2)=\kappa_f^2(s_1)\kappa'_g(s_2)$, then $\kappa_f''(s_1)\kappa_g^3(s_2)\neq\kappa_f^3(s_1)\kappa''_g(s_2)$.
    \item If $f(s_1),f(s_2)$ is a parallel pair and $\kappa_f(s_1)-\operatorname{sgn}\left(\left<f'(s_1),f'(s_2)\right>\right)\kappa_f(s_2)= 0.$ Then  $\kappa_f'(s_1)\kappa_g^2(s_2)\neq \kappa_f^2(s_1)\kappa'_g(s_2)$.
    \item If $f(s_1),f(s_2)$ is a parallel pair, then the conditions $\kappa_f'(s_1)\kappa_g^2(s_2)=\kappa_f^2(s_1)\kappa'_g(s_2)$ and $\det\big(f(s_1)-f(s_2),f'(s_1)\big)=0$ do not occur simultaneously. 
    \item If $f(s_1),f(s_2)$ and $f(s_3),f(s_4)$ are two distinct parallel pairs which have the same affine chord, then $\kappa_f(s_1)-\operatorname{sgn}\left(\left<f'(s_1),f'(s_2)\right>\right)\kappa_f(s_2)\neq 0$ or  $\kappa_f(s_3)-\operatorname{sgn}\left(\left<f'(s_3),f'(s_4)\right>\right)\kappa_f(s_4)\neq 0$.
    \item If $f(s_1),f(s_2)$ and $f(s_3),f(s_4)$ are two distinct parallel pairs such that $\kappa_f(s_1)-\operatorname{sgn}\left(\left<f'(s_1),f'(s_2)\right>\right)\kappa_f(s_2)\neq 0,$  $\kappa_f(s_3)-\operatorname{sgn}\left(\left<f'(s_3),f'(s_4)\right>\right)\kappa_f(s_4)\neq 0,$ and 
    \scriptsize
    $$\frac{\kappa_f(s_1)f(s_1)-\operatorname{sgn}\left(\left<f'(s_1),f'(s_2)\right>\right)\kappa_f(s_2)f(s_2)}{\kappa_f(s_1)-\operatorname{sgn}\left(\left<f'(s_1),f'(s_2)\right>\right)\kappa_f(s_2)}=
    \frac{\kappa_f(s_3)f(s_3)-\operatorname{sgn}\left(\left<f'(s_3),f'(s_4)\right>\right)\kappa_f(s_4)f(s_4)}{\kappa_f(s_3)-\operatorname{sgn}\left(\left<f'(s_3),f'(s_4)\right>\right)\kappa_f(s_4)},$$
    \normalsize
    then
    $\det\left(f(s_2)-f(s_1),f(s_4)-f(s_3)\right)\neq 0,$
    $\kappa_f'(s_1)\kappa_f^2(s_2)\neq\kappa_f'(s_2)\kappa_f^2(s_1),$
    and 
    $\kappa_f'(s_3)\kappa_f^2(s_4)\neq\kappa_f'(s_4)\kappa_f^2(s_3).$
\end{enumerate}
In the above properties $'$ denote the derivative with respect to the arc length parameter.

Then the set $\mathcal{G}$ is a generic subset of $C^{\infty}(S^1,\mathbb{R}^2)$. 

Furthermore, if $f\in\mathcal{G},$ then the Centre Symmetry Set of $f$ has the following properties.
\begin{enumerate}[(a)]
    \item It is the union of smooth curves with at most cusp singularities.
    \item It has only transversal self-crossings.
    \item It has inflexion points only on double tangents of $f$.
    \item No asymptote of the $\CSS$ is a double asymptote.
    \item In a neighbourhood of an asymptote the curvature function of the $\CSS$ changes sign. 
\end{enumerate}
\end{thm}
\begin{proof}
We will show that each property is generic by the Thom Transversality Theorems -- the one for maps and the one for multijets (see Theorems 4.9 and 4.13 in \cite{GG-Book}). We will denote by $J^k_s(S^1,\mathbb{R}^2)$ the \textit{$s$-fold $k$-jet bundle} and let $(S^1)^{(k)}$ denote the set $\left(S^1\right)^{\!k}\,\setminus\,\left\{(s_i)_{i=1}^{k}\ |\ s_i\neq s_j\text{ if }i\neq j\right\}$.

The genericity of properties (i), (ii), and (v) are proved in \cite{Romero} (see Theorem 2.12) using Thoms Transversality Theorems, too.

We will show that Property (iii) is generic. Let $f\in C^{\infty}(S^1,\mathbb{R}^2)$ be regular. It is easy to calculate that the transversality of the map $j^2_2f:\left(S^1\right)^{(2)}\to J_2^2(S^1,\mathbb{R}^2)$ to the submanifold
\begin{align*}
    \Big\{\left(j^2g(s_1),j^2h(s_2)\right)\in J^2_2(S^1,\mathbb{R}^2)\ \Big|\ &g'(s_1)\neq 0, h'(s_2)\neq 0, \det\left(g'(s_1),h'(s_2)\right)=0, \\  &\det\left(g(s_1)-h(s_2),g'(s_1)\right)=0\Big\}
\end{align*}
is equivalent to $\kappa_g(s_1)\kappa_h(s_2)\neq 0$. Therefore, Property (iii) is generic by the Thom Transversality Theorem. In a similar way one can prove genericity of Properties (iv), (v), and (vii).

Now, let us deal with genericity of Property (vi). It directly follows from the Thom Transversality Theorem and the transversality of the map $j^4_2f:(S^1)^{(2)}\to J^4_2(S^1,\mathbb{R}^2)$ to the submanifold
\begin{align*}
    \Big\{\left(j^4g(s_1),j^4h(s_2)\right)\in J^4(S^1,\mathbb{R}^2)\,\Big|\, g'(s_1)\neq 0, h'(s_2)\neq 0, \det\left(g'(s_1),h'(s_2)\right)=0, \\ 
    \kappa_f'(s_1)\kappa_g^2(s_2)-\kappa_g'(s_2)\kappa_f^2(s_1)=0\}.
\end{align*}

The set 
\small
\begin{align*}
    S_{\text{vi}}:=\Big\{\left(j^3g(s_1),j^3h(s_2)\right)\in J^3(S^1,\mathbb{R}^2)\,\Big|\, g'(s_1)\neq 0, h'(s_2)\neq 0, \det\left(g'(s_1),h'(s_2)\right)=0, \\ 
    \kappa_f'(s_1)\kappa_g^2(s_2)-\kappa_g'(s_2)\kappa_f^2(s_1)=0,  \det\left(g(s_1)-h(s_2),g'(s_1)\right)=0\Big\}.
\end{align*}
is a submanifold of codimension $3$ in $J^3(S^1,\mathbb{R}^2)$. If $j^3f$ is transversal to $S_{\text{vi}}$, then $j^3f$ does not intersect $S_{\text{iv}}$. 

Finally, genericity of Property (vii) follows from the  Thom Transversality Theorem and the transversality of the map $j_4^1f:\left(S^1\right)^{(4)}\to J_4^1\left(S^1,\mathbb{R}^2\right)$ to the following submanifold:
\begin{align*}
    \Big\{&\left(j^1g_k(s_k)\right)_{k=1}^4\in J_4^1\left(S^1,\mathbb{R}^2\right)\,\Big|\, g_k'(s_k)\neq 0\text{ for } k=1,2,3,4, \\ &\det\left(g_1'(s_1),g_2'(s_2)\right)=0, \det\left(g_3'(s_3),g_4'(s_4)\right)=0, \\ &\det\left(g_2(s_2)-g_1(s_1),g_3(s_3)-g_1(s_1)\right)=0, \det\left(g_2(s_2)-g_1(s_1),g_4(s_4)-g_1(s_1)\right)=0\Big\}.
\end{align*}
In a similar way we can prove genericity of the last property. 

Now let's assume that $f$ is in the generic set $\mathcal{G}$ satisfying properties (i-viii).

Note that Property (a) follows From Properties (i-iv) -- see the local parameterization of the Centre Symmetry Set (we need to have at least one non-inflexion point in each parallel pair) and Lemma \ref{CuspsofCSS}. Property (b) is true by virtue of Property (viii). Then, Property (c) follows directly from Property (vi) and Proposition \ref{PropCurvCss}, and the property (d) from Property (vii). Finally, by Property (v) we obtain Property (e), which ends the proof.
\end{proof}

\begin{prop}[Proposition 2.13 in \cite{Romero}]\label{prop:finite_rot_num} Let $M$ be a smooth regular closed curve. If $M$ has only non-degenerate inflexion points and no undulation points, then the rotation number of $M$ and the number of inflexion points of $f$ are finite.
\end{prop}

\begin{remark}
Since generic curves satisfy the assumptions of Proposition \ref{prop:finite_rot_num}, the rotation number and the number of inflexion points of each generic curve are finite.
\end{remark}

\section{A Decomposition of a Curve into Parallel Arcs}\label{decomp}

Let's assume that $M$ is a generic regular closed curve. In this section we will adapt an algorithm of decomposition of $M$ into parallel arcs (for further details see Section $3$ in \cite{ZD_Wigner}) that will help us to study the geometry of Centre Symmetry Set of $M$.

We already know that if $M$ is a generic regular closed curve, then $\CSS(M)$ is a union of smooth parameterized curves and double tangents. 
Each of the smooth parameterized curves belonging to $\CSS(M)$ will be called an \textit{smooth semi-branch} of $\CSS(M)$ (see Figure \ref{FigureAlmostSmoothBranches}).

\begin{figure}[h]
\centering
\includegraphics[scale=0.35]{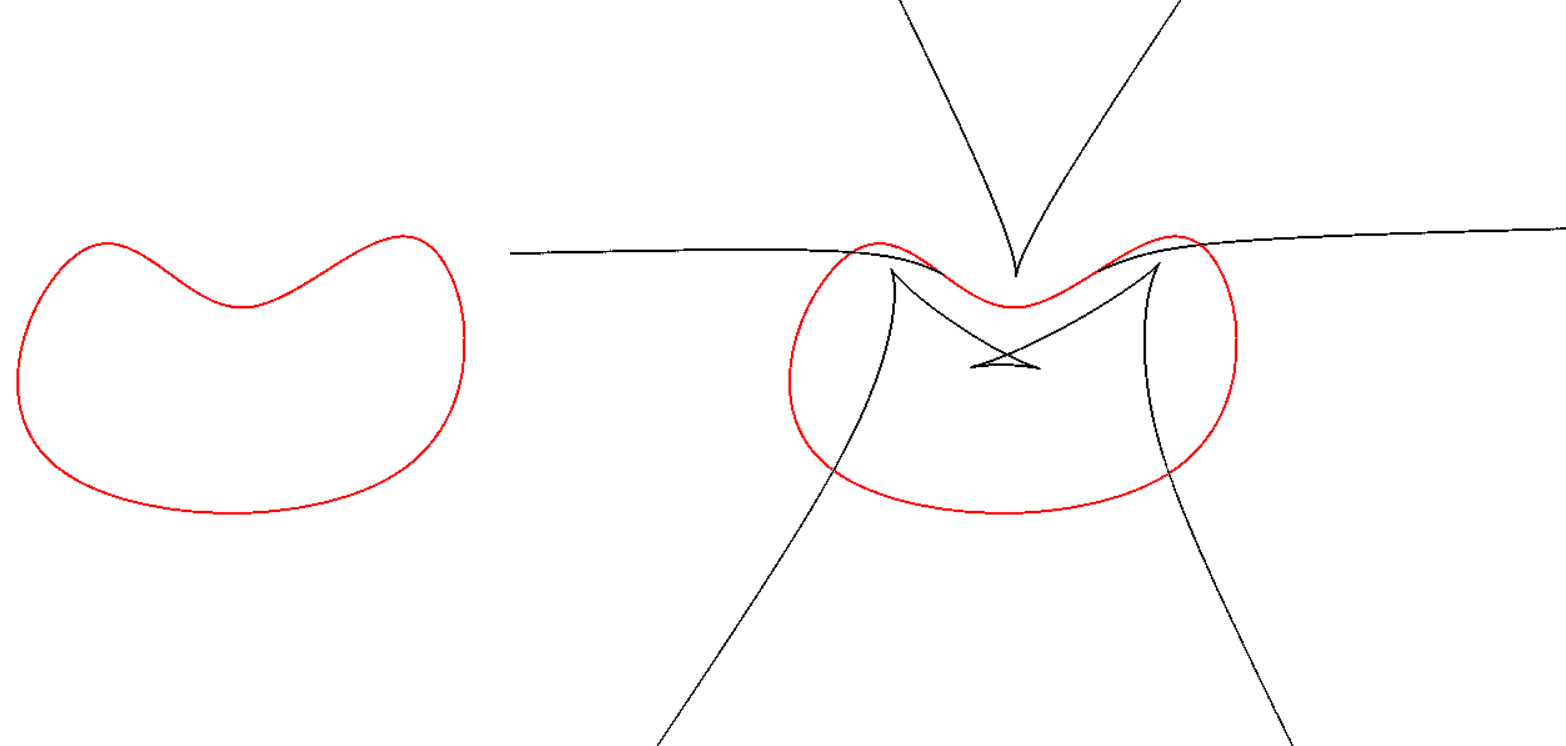}
\caption{A curve $M$ on the left, and $M$ together with $4$ smooth semi-branches of $\CSS(M)$ on the right}
\label{FigureAlmostSmoothBranches}
\end{figure}

Notice that some of the smooth semi-branches might have asymptotes. 
Each of the asymptotes is either a bi-asymptote of one smooth semi-branch or a single asymptote of two different smooth semi-branches. 
When we use the term \textit{bi-asymptote}, we mean that there are exactly two parts of the curve approaching an asymptote.

Smooth semi-branches that have the same asymptote we identify into one smooth semi-branch.
Since $M$ is generic, it follows that we have a finite number of smooth semi-branches. 
After all possible identifications are done, the elements of the resulting set of smooth semi-branches will be called \textit{smooth branches} of $\CSS(M)$.

In Figure \ref{FigureAlmostSmoothBranches} we illustrate a non-convex curve $M$, $\CSS(M)$, and two different smooth branches of $\CSS(M)$ (one smooth branch is created from the left and right smooth semi-branch, the other smooth branch is created from the bottom and top smooth semi-branch).

Informally, we can say that asymptotes connect almost-smooth branches into smooth branches.

\begin{definition}\label{DefAngleFunctin}
Let $S^1\ni s\mapsto f(s)\in\mathbb{R}^2$ be a parameterization of a smooth closed curve $M$, such that $f(s_0)$ is not an inflexion point. A function $\varphi_{M}:S^1\to [0,\pi]$ is called an \textit{angle function of} $M$ if $\varphi_{M}(s)$ is the oriented angle between $f'(s)$ and  $f'(s_0)$ modulo $\pi$. We identify the set $[0,\pi]$ modulo $\pi$ with $S^1$.
\end{definition}

\begin{definition}\label{DefLocalExtrema}
A point $\varphi$ in $S^1$ is a \textit{local extremum }of $\varphi_{M}$ if there exists $s$ in $S^1$ such that $\varphi_{M}(s)=\varphi$, $\varphi'_{M}(s)=0$, $\varphi''_{M}(s)\neq 0.$ The local extremum $\varphi$ of $\varphi_{M}$ is a \textit{local maximum (respectively minimum)} if $\varphi''_{M}(s)<0$ (respectively $\varphi''_{M}>0$). We denote by $\mathcal{M}(\varphi_{M})$ the set of local extrema of $\varphi_{M}$.  
\end{definition}

In what follows let $f$ be the arc length parameterization of $M$ and let $\varphi_{M}$ be the angle function of $M$.
Then, we can quite easily obtain the following properties of the angle function.

\begin{prop}\label{PropAlgAngleFun}
For $M,$ $f,$ and $\varphi_{M}$ defined above the following conditions are true
\begin{enumerate}[(i)]
\item $f(s_1),$ $f(s_2)$ is a parallel pair of $M$ if and only if $\varphi_{M}(s_1)=\varphi_{M}(s_2)$,
\item $\varphi_{M}'(s)$ is equal to the signed curvature of $M$ with respect to $f,$
\item $M$ has an inflexion point at $f(s_0)$ if and only if $\varphi_{M}(s_0)$ is a local extremum,
\item if $\varphi_{M}(s_1),$ $\varphi_{M}(s_2)$ are consecutive local extrema, then one of them is a local maximum and the other one is a local minimum.
\end{enumerate}
\end{prop}

\begin{lemma}[Lemma 3.4 in \cite{ZD_Wigner}] \label{LemmaAlgEvenNumOfInfl}
The function $\varphi_{M}$ has an even number of local extrema. In other words, $M$ has an even number of inflexion points.
\end{lemma}

\begin{definition}\label{DefSeqOfLocExtr}
The \textit{sequence of local extrema} is the sequence $(\varphi_0, \varphi_1, \ldots, \varphi_{2n-1})$ where $\{\varphi_0, \varphi_1, \ldots, \varphi_{2n-1}\}=\mathcal{M}(\varphi_{M})$ and the order is compatible with the orientation of $S^1=\varphi_{M}(S^1)$.
\end{definition} 

\begin{definition}\label{DefSeqOfParallPts}
The \textit{sequence of division points} $\mathcal{S}_{M}$ is the following sequence $(s_0, s_1, \ldots, s_{k-1})$, where $\{s_0, s_1, \ldots, s_{k-1}\}=\varphi_{M}^{-1}\left(\mathcal{M}(\varphi_{M})\right)$ if $\mathcal{M}(\varphi_{M})$ is not empty, otherwise $\{s_0, s_1, \ldots, s_{k-1}\}=\varphi_{M}^{-1}\left(\varphi_{M}(s_0)\right),$ where $s_0$ is fixed. The order of $\mathcal{S}_{M}$ is compatible with the orientation of $M$.   
\end{definition}

Let $M$ have inflexion points. If $s_k$ belongs to the sequence of division points, then $f(s_k)$ is an inflexion point or a point which is parallel to an inflexion point. 

In the case when $M$ has no inflexion points, the sequence of division points consists of a fixed $s_0$ and points $s_k$ such that $f(s_0),$ $f(s_k)$ are parallel pairs.

\begin{definition}
Two arcs $\mathcal{A}_1,$ $\mathcal{A}_2$ of a curve $M$ are called \textit{parallel arcs} if there exists a bijection $\mathcal{A}_1\ni p\mapsto\omega(p)\in\mathcal{A}_2$ such that $p,\omega(p)$ is a parallel pair of $M.$ 
\end{definition}

We can see that the curve $M$ is divided into arcs by the images of points in the sequence of division points. The set of these arcs splits into subsets contained of mutually parallel arcs
(see Definition \ref{DefSetParallArcs}).
Figure \ref{PictureAlgFindingPhiSets}  illustrates a closed regular curve $M$, the sequence of division points, and the angle function $\varphi_{M}$.

\begin{figure}[h]
\centering
\includegraphics[scale=0.35]{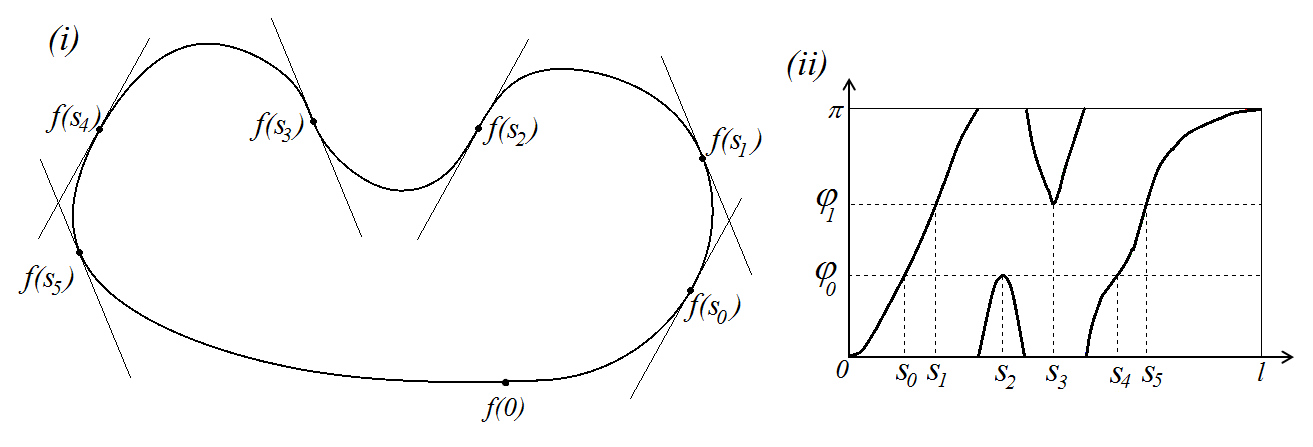}
\caption{(i) A closed regular curve $M$ with points $f(s_i)$ and lines tangent to $M$ at these points, (ii) a graph of the angle function $\varphi_{M}$ with values of $\varphi_i$ and $s_i$ (\cite{ZD_Wigner})}
\label{PictureAlgFindingPhiSets}
\end{figure}

\begin{prop}[Proposition 3.7 in \cite{ZD_Wigner}]\label{PropParalllel-infexion}
If $M$ is a generic regular closed curve and $a\in M$ is an inflexion point then the number of points $b\in M$, such that $b\neq a$ and $a$, $b$ is a parallel pair, is even.  
\end{prop}

The Proposition \ref{PropAlgAngleFun} (iv) implies that the number of inflexion points of a generic regular closed curve is even. Thus, using Proposition \ref{PropParalllel-infexion} we have the following corollary.

\begin{corollary}\label{PropMisEven}
If $M$ is a generic regular closed curve then the number of elements in the sequence $\mathcal{S}_M$ is even.
\end{corollary}

As a consequence of the Corollary \ref{PropMisEven} and since we assumed that $M$ is a generic regular closed curve, we can denote the number of elements in the sequence of division points as $\#\mathcal{S}_{M}=2m$.

Now, we can define functions $\minm_{2m},$  $\maxm_{2m}:\{0,1,\ldots,2m-1\}^2\to\{0,1,\ldots,2m-1\}$ as
\begin{align*}
\minm_{2m}(k,l):&=\left\{\begin{array}{ll}2m-1, &\text{ if }\{k,l\}=\{0, 2m-1\},\\ \min(k,l), &\text{ otherwise},\end{array}\right.\\
\maxm_{2m}(k,l):&=\left\{\begin{array}{ll}0, &\text{ if }\{k,l\}=\{0, 2m-1\},\\ \max(k,l), &\text{ otherwise}.\end{array}\right.
\end{align*}
As we can see, $\minm_{2m}$ is an analog of the minimum function modulo $2m$, and $\maxm_{2m}$ is an analog of the maximum function modulo $2m.$
Let $L_{M}$ be the length of $M.$ Then, we denote  an interval $(s_{2m-1}, L_{M}+s_0)$ by  $(s_{2m-1},s_0).$

\begin{definition}\label{DefSetParallArcs}
If $\mathcal{M}(\varphi_{M})=\{\varphi_0, \varphi_1, \ldots, \varphi_{2n-1}\}$, then for every $i\in\{0, 1, \ldots, 2n-1\}$, the \textit{set of parallel arcs} $\Phi_i$ is the following set
\begin{align*}
\Phi_i=\Big\{\overarc{\p_k}{\p_l}\ \big|\ &k-l=\pm 1\ \mbox{mod} (2m),\ \varphi_{M}(s_k)=\varphi_i,\ \varphi_{M}(s_l)=\varphi_{i+1}, \\ &\varphi_{M}\big((s_{\minm_{2m}(k,l)}, s_{\maxm_{2m}(k,l)})\big)=(\varphi_i, \varphi_{i+1})\Big\},
\end{align*}
where $\p_k=f(s_k)$ and $\overarc{\p_k}{\p_l}=f\left(\big[s_{\minm_{2m}(k, l)}, s_{\maxm_{2m}(k, l)}\big]\right)$.

\noindent If $\mathcal{M}(\varphi_{M})$ is empty then we define only one \textit{set of parallel arcs} as follows:
\begin{align*}
\Phi_0=\big\{\overarc{\p_0}{\p_1},  \overarc{\p_1}{\p_2}, \ldots,  \overarc{\p_{2m-2}}{\p_{2m-1}},  \overarc{\p_{2m-1}}{\p_0}\big\}.
\end{align*}
\end{definition}

In the above definition indexes $i$ in $\varphi_i$ are computed modulo $2n$, indexes $j, j+1$ in $\overarc{\p_{j}}{\p_{j+1}}$ and $\overarc{\p_{j+1}}{\p_{j}}$ are computed modulo $2m$.
The set of parallel arcs has the following property.

\begin{prop}\label{PropDiffeoBetweenParallelArcs}
Let $f:S^1\to\mathbb{R}^2$ be the arc length parameterization of $M$.
For every two arcs $\overarc{\p_{k}}{\p_{l}}$, $\overarc{\p_{k'}}{\p_{l'}}$ in $\Phi_i$ the well defined map
\begin{align*}
\overarc{\p_{k}}{\p_{l}}\ni p\mapsto P(p)\in\overarc{\p_{k'}}{\p_{l'}},
\end{align*}
where the pair $p, P(p)$ is a parallel pair of $M$, is a diffeomorphism.
\end{prop}

\begin{definition}\label{DefGlueingScheme}
Let $\overarc{\p_{k_1}}{\p_{k_2}}$, $\overarc{\p_{l_1}}{\p_{l_2}}$ belong to the same set of parallel arcs, then $\begin{array}{ccc}
\p_{k_1} &\frown&\p_{k_2} \\ \hline
\p_{l_1} &\frown&\p_{l_2} \\ \hline \end{array}$ denotes the following set (the \textit{arc}) 
\begin{align*}
\mbox{cl}\Big\{(a,b)\in M\times M\ \Big|\ &a\in\overarc{\p_{k_1}}{\p_{k_2}}, b\in\overarc{\p_{l_1}}{\p_{l_2}},\ a,b\text{ is a parallel pair of } M \Big\}.
\end{align*}

Additionally we will denote $\displaystyle\bigcup_{i=1}^{n-1}\begin{array}{ccc} 
\p_{k_i} &\frown &\p_{k_{i+1}} \\ \hline
\p_{l_i} &\frown &\p_{l_{i+1}} \\ \hline \end{array}$ by
$\begin{array}{ccccc}
\p_{k_1} &\frown &\ldots &\frown & \p_{k_n}\\ \hline
\p_{l_1} &\frown &\ldots &\frown& \p_{l_n}\\ \hline \end{array}.$
This set will be called a \textit{glueing scheme}.
\end{definition}

\begin{remark}\label{RemInflexionInScheme}
If $\overarc{\p_k}{\p_l}$ belongs to the set of parallel arcs, then  there are neither inflexion points nor points with tangent lines parallel to tangent lines at inflexion points of $M$ in $\overarc{\p_k}{\p_l}\setminus\{\p_k, \p_l\}$.
\end{remark}

\begin{definition}
Let $a, b \in M$ be a standard pair of points. If $\kappa_M(a) + \kappa_M(b) \neq 0,$ then we define the $\CSS\text{-}\textit{point map}$ as the map
\begin{align*}
\pi_{\CSS}(a,b) = 
\dfrac{\kappa_M(a) \cdot a + \kappa_M(b) \cdot b}{\kappa_M(a) + \kappa_M(b)}.
\end{align*}
\end{definition}

If $\kappa_M(a) + \kappa_M(b) = 0,$ the $\CSS\text{-}$point map would produce an asymptote of $\CSS(M)$ (see Remark \ref{rem:asymptote}).

Let $\mathcal{A}_1=\overarc{\p_{k_1}}{\p_{k_2}}$ and $\mathcal{A}_2=\overarc{\p_{l_1}}{\p_{l_2}}$ be two arcs of $M$ which belong to the same set of parallel arcs. Then $\CSS \Big(\mathcal{A}_1\cup\mathcal{A}_2\Big) = \pi_{\CSS} \left( \, \begin{array}{ccc}
\p_{k_1} &\frown&\p_{k_2} \\ \hline
\p_{l_1} &\frown&\p_{l_2} \\ \hline \end{array} \, \right)$ (see Figure \ref{FigEqFromParallelArcs}). 

\begin{figure}[h]
\centering
\includegraphics[scale=0.45]{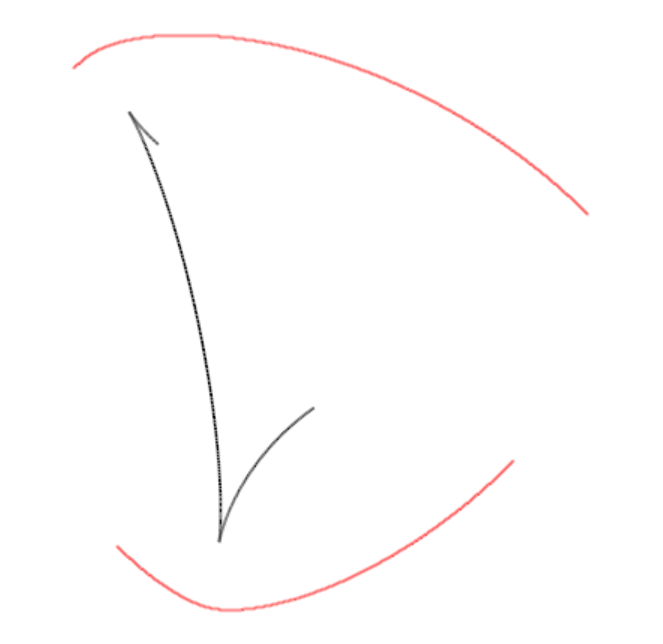}
\caption{Two parallel arcs $\mathcal{A}_1,\mathcal{A}_2$ and $\CSS(\mathcal{A}_1\cup\mathcal{A}_2)$}
\label{FigEqFromParallelArcs}
\end{figure}

As a consequence of this note, we obtain Proposition \ref{PropNumDiffArcs}.

\begin{prop}\label{PropNumDiffArcs}
The Centre Symmetry Set $\CSS(M)$ is the image of the union of \linebreak 
$\displaystyle \sum_{i}{\#\Phi_i\choose 2}$ different arcs under the $\CSS\text{-}\text{point map}$ $\pi_{\CSS}$.
\end{prop}

\begin{prop}[Proposition 3.15 in \cite{ZD_Wigner}]\label{PropAlgAlwaysGoFurhter}
Let $M$ be a generic regular closed non-convex curve. If a glueing scheme is of the form $\begin{array}{ccc} 
\p_{k_1} &\frown &\p_{k_2} \\ \hline
\p_{l_1} &\frown &\p_{l_2} \\ \hline \end{array}$, then this scheme can be prolonged in a unique way to $\begin{array}{ccccc} 
\p_{k_1} &\frown &\p_{k_2} &\frown &\p_{k_3}\\ \hline
\p_{l_1} &\frown &\p_{l_2} &\frown &\p_{l_3}\\ \hline \end{array}$ such that $(k_1,l_1)\ne (k_3,l_3)$.
\end{prop}

\begin{remark}
To avoid repetition in the union $\displaystyle\bigcup_{i=1}^{n-1}\begin{array}{ccc} 
\p_{k_i} &\frown &\p_{k_{i+1}} \\ \hline
\p_{l_i} &\frown &\p_{l_{i+1}} \\ \hline \end{array}$ in Definition \ref{DefGlueingScheme} we assume that no pair $\begin{array}{c} \p_k \\ \hline \p_l \\ \hline \end{array}$ except the beginning and the end can appear twice in the glueing scheme. 
Furthermore, if the pair $\begin{array}{c}\p_k \\ \hline \p_l \\ \hline \end{array}$ appears in the glueing scheme, then the pair $\begin{array}{c} \p_l \\ \hline \p_k \\ \hline \end{array}$ does not appear unless they are the beginning and the end of the scheme.
\end{remark}

The image of a glueing scheme under the $\CSS\text{-}$point map $\pi_{\CSS}$ represents a part of a branch of the Centre Symmetry Set. If we endow the set of all possible glueing schemes with the inclusion relation, then this set is partially ordered.

The uniqueness of prolongation of the glueing set (see Proposition \ref{PropAlgAlwaysGoFurhter}) leads to the following proposition.

\begin{prop}\label{PropChains}
The set of all glueing schemes endowed with the inclusion relation is the disjoint union of totally ordered sets.
\end{prop}

There is only finite number of arcs from which we can construct branches of $\CSS(M)$. Therefore, we can define a maximal glueing scheme.

\begin{definition}
A \textit{maximal glueing scheme} is a glueing scheme which is a maximal element of the set of all glueing schemes endowed with the inclusion relation.
\end{definition}

Notice that starting from Proposition \ref{PropAlgAlwaysGoFurhter} we considered only non-convex curves. In the following remark we will examine the case of convex curves.

\begin{remark}\label{RemGlueSchemeOval}
If $M$ is a generic regular convex curve, then the set of parallel arcs is equal to $\Phi_0=\{\overarc{\p_0}{\p_1}, \overarc{\p_1}{\p_0}\}$. As a result, the only maximal glueing scheme is 
$\begin{array}{ccc}
\p_0&\frown& \p_1\\ \hline
\p_1&\frown& \p_0\\ \hline\end{array}.$
\end{remark}

\begin{lemma}[Lemma 3.20 in \cite{ZD_Wigner}]\label{LemPropMaxGlueSchemes}
Let $f:S^1\mapsto\mathbb{R}^2$ be the arc length parameterization of $M$. Then
\begin{enumerate}[(i)]
\item for every two different arcs $\overarc{\p_{k_1}}{\p_{k_2}}$, $\overarc{\p_{l_1}}{\p_{l_2}}$ in $\Phi_i$ there exists exactly one maximal glueing scheme containing 
$\begin{array}{ccc}
\p_{k_1} &\frown &\p_{k_2} \\ \hline
\p_{l_1} &\frown &\p_{l_2} \\ \hline \end{array},$
$\begin{array}{ccc}
\p_{k_2} &\frown &\p_{k_1} \\ \hline
\p_{l_2} &\frown &\p_{l_1} \\ \hline \end{array},$
\linebreak 
$\begin{array}{ccc}
\p_{l_1} &\frown &\p_{l_2} \\ \hline
\p_{k_1} &\frown &\p_{k_2} \\ \hline \end{array},$
or
$\begin{array}{ccc} 
\p_{l_2} &\frown &\p_{l_1} \\ \hline
\p_{k_2} &\frown &\p_{k_1} \\ \hline \end{array}$.

\item every maximal glueing scheme is in the following form $\begin{array}{ccccc}
\p_k &\frown&\ldots&\frown&\p_{k'} \\ \hline
\p_l &\frown&\ldots&\frown&\p_{l'} \\ \hline \end{array}$, where $\{p_k, p_l\}=\{p_{k'},p_{l'}\}$ whenever $p_k\neq p_l$ and $p_{k'}\neq p_{l'}$.

\item if $\p_k$ is an inflexion point of $M$, then there exists a maximal glueing scheme which is in the form 
\begin{align*}
\begin{array}{ccccccccc}
\p_k &\frown&\p_{k_1}&\frown&\ldots&\frown&\p_{k_n}&\frown&\p_l \\ \hline
\p_k &\frown&\p_{l_1}&\frown&\ldots&\frown&\p_{l_n}&\frown&\p_l \\ \hline \end{array},
\end{align*}
where $\p_l$ is a different inflexion point of $M$ and $p_{k_i}\neq p_{l_i}$ for $i=1, 2, \ldots, n$.
\end{enumerate}
\end{lemma}

By Proposition \ref{PropNumDiffArcs} and Lemma \ref{LemPropMaxGlueSchemes} we obtain the following theorem.

\begin{thm}\label{ThmGlueSchemeIsBranch} 
The image of every maximal glueing scheme of $M$ under the $\CSS\text{-}$point map $\pi_{\CSS}$ is a smooth branch of the Centre Symmetry Set of $M$ and all smooth branches of the Centre Symmetry Set can be obtained this way.
\end{thm}

\section{Global Geometry of the Centre Symmetry Set}

In this section we study global properties of the Centre Symmetry Set which follow from maximal glueing schemes introduced in Section \ref{decomp} and methods similar to the ones presented in \cite{Romero, D-Z3, ZD_Wigner}.

Theorem \ref{ThmGlueSchemeIsBranch} along with Lemma \ref{LemPropMaxGlueSchemes} yield the following fact.

\begin{prop}\label{PropAlgPartBetweenIflPt}
Let $M$ be a generic regular closed curve. If $M$ has $2n$ inflexion points then there exist exactly $n$ smooth branches of  $\CSS(M)$ connecting pairs of inflexion points on $M$ and every inflexion point of $M$ is the end of exactly one branch of $\CSS(M)$. Other branches of $\CSS(M)$ are closed curves.
\end{prop}

\begin{proof}
If $M$ has $2n$ inflexion points, then the maximal glueing scheme presented in Section \ref{decomp} produces maximal glueing schemes that begin and end with inflexion points (see Lemma \ref{LemPropMaxGlueSchemes}).
The uniqueness of prolongation of the maximal glueing scheme (see Proposition \ref{PropAlgAlwaysGoFurhter}) provides that there are exactly $n$ of such maximal glueing schemes.
Then, Theorem \ref{ThmGlueSchemeIsBranch} allows us to obtain exactly $n$ smooth branches of the Centre Symmetry Set that begin and end with inflexion points.

Directly from Lemma \ref{LemPropMaxGlueSchemes} (ii) we obtain that the other maximal glueing schemes begin and end with the same parallel pair, which means that the corresponding smooth branches of $\CSS(M)$ are closed.
\end{proof}

Now we will consider the secant caustic and its relationship with $\CSS(M).$

\begin{definition}
The \textit{secant caustic} of $M$ is the following set:
\begin{align*}
\SC(M)=\textrm{cl}\left\{a-b\ \big|\ a,b \text{ is a parallel pair of } M\right\}.
\end{align*}
\end{definition}

\noindent The geometry of the secant caustic and its properties are further analysed in \cite{Romero}. We will now focus on properties connected with the Centre Symmetry Set.

\begin{remark}[Remark 2.20 in \cite{Romero}]\label{SCvsCSS}
A parallel pair $a, b$ of a curve $M$ produces an asymptote of the Centre Symmetry Set $\CSS(M)$ if and only if it gives rise to a singular point of the secant caustic.
\end{remark}

\begin{figure}[h]
\centering
\includegraphics[scale=0.30]{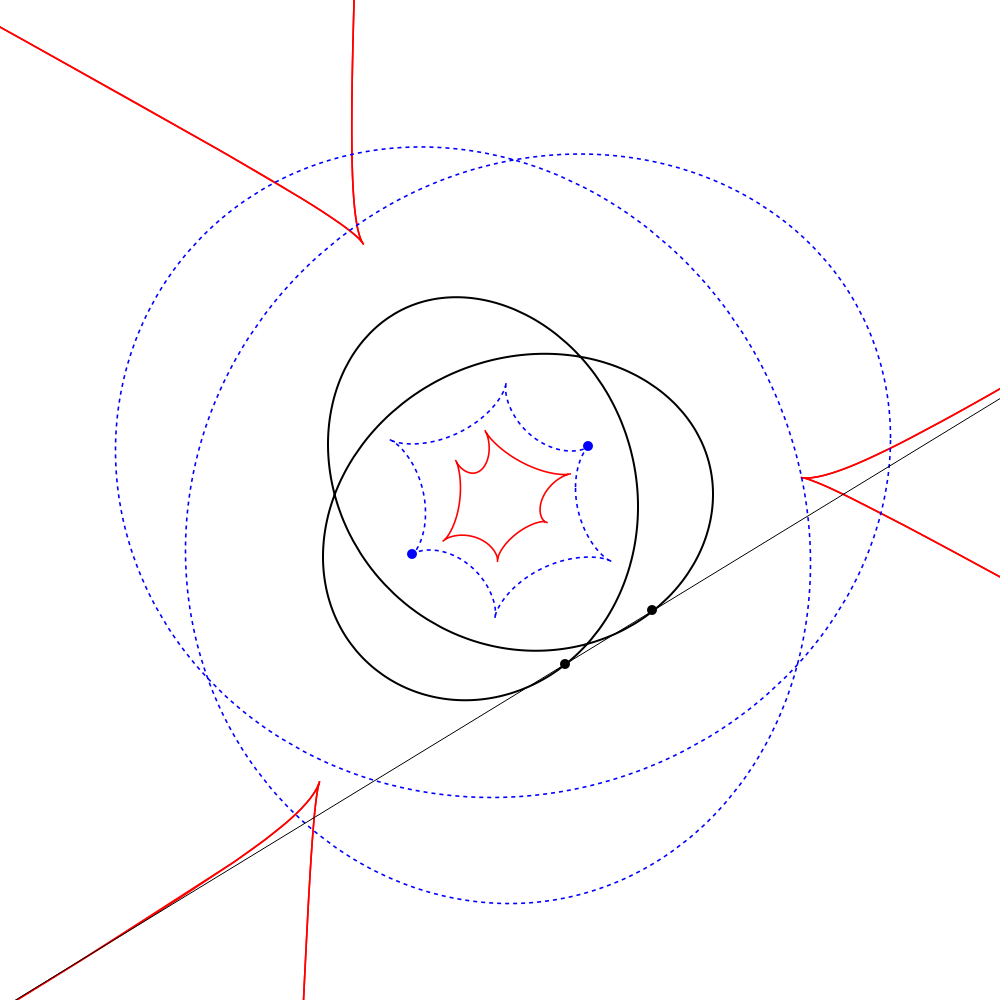}
\caption{A $2$-rosette $M$, $\CSS(M),$ and $\SC(M)$ (the dashed line)}
\label{FigCssSc}
\end{figure}

\begin{definition}\label{rosette}
An $n$-\textit{rosette} is a closed regular curve with non-vanishing curvature with the rotation number equal to $n.$
\end{definition}

In Figure \ref{FigCssSc} we present a $2$-rosette $M$ with a parameterization
\begin{align*}
    g(t)=\left(p(t)\cos t-p'(t)\sin t,p(t)\sin t+p'(t)\cos t\right),
\end{align*}
where $p(t)=14+3\cos\frac{3t}{2}+\frac{1}{5}\sin\frac{5t}{2}$ is so-called support function of a rosette (for further details on support functions see e.g. \cite{Zwierz3}). We illustrate a parallel pair $g(t_0),g(t_0+2\pi),$ where $t_0\approx 5.38207$ that produces an asymptote of the Centre Symmetry Set of $M.$ Moreover two cusps of the secant caustic of $M$ are marked. Although it may seem that the pair $g(t_0),g(t_0+2\pi)$ creates a double tangent -- it does not. The double tangent is created by the tangent line to $g(t_1)$ and $g(t_1+2\pi),$ where $t_1\approx 5.26053$ (it is not marked on the diagram). In a generic situation (see Theorem \ref{ThmGenericCSS}) a double tangent line of $M$ cannot be an asymptote of the Centre Symmetry Set of $M.$

In what follows, we will denote the translation by a vector $v\in\mathbb{R}^2$ as $\tau_{v}.$

\begin{definition}\label{DefCurved}
Let $a,$ $b$ be a parallel pair of $M$ and assume that $a$ and $b$ are not inflexion points. We say that $M$ is \textit{curved in the same side at $a$ and $b$} (respectively \textit{curved in the different sides at $a$ and $b$}) if the germs of the curves $M$ and $\tau_{a-b}(M)$ at $a=\tau_{a-b}(b)$ are on the same side (respectively on the different sides) of the tangent line to $M$ at $a.$
\end{definition}

\noindent We illustrate Definition \ref{DefCurved} in Figure \ref{FigCurved}.

\begin{figure}[h]
\centering
\includegraphics[scale=0.47]{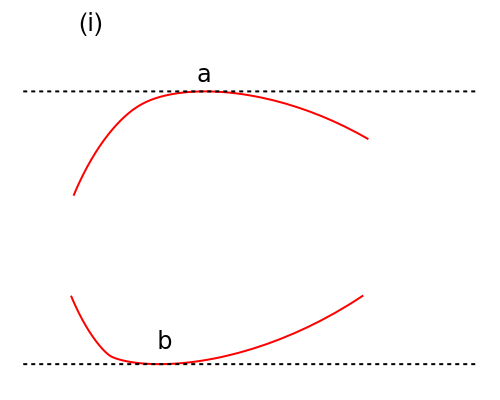}
\includegraphics[scale=0.47]{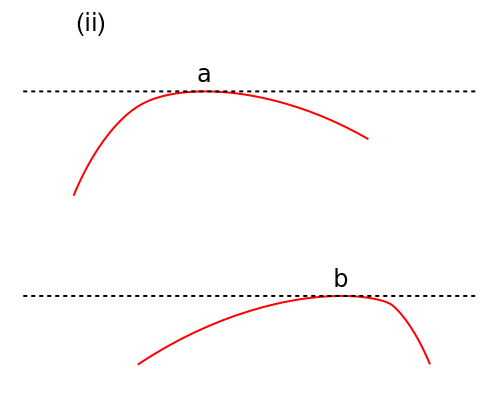}
\caption{(i) A curve curved in the different sides at a parallel pair $a,$ $b,$ (ii) a curve curved in the same side at a parallel pair $a,$ $b$}
\label{FigCurved}
\end{figure}

\begin{lemma}\label{CurvedVsAsymptotes}
Let $a,$ $b$ be a parallel pair of $M$ and assume that $a$ and $b$ are not inflexion points.
If $M$ is curved in the different sides at $a$ and $b,$ then $a$ and $b$ cannot produce an asymptote of $\CSS(M).$
\end{lemma}

\begin{proof}
Let $a, b \in M$ be a standard pair of points.
When $M$ is curved in the different sides at $a$ and $b,$ we know that the curvatures of $M$ at points $a$ and $b$ have the same sign. Therefore, their sum cannot be equal to zero. By Remark \ref{rem:asymptote}, this means that $a$ and $b$ cannot produce an asymptote of $\CSS(M).$
\end{proof}

Using the Definition \ref{DefCurved} we can formulate and prove the following theorem.

\begin{thm}\label{ThmCurvPosWrinkled} 
Let $\mathcal{P},$ $\mathcal{Q}$ be curves with end points $p_0,$ $p_1$ and $q_0,$ $q_1$ respectively. We make the following assumptions.
\begin{enumerate}[(i)]
\item The points $p_i,$ $q_i$ form a parallel pair for $i=0,1.$
\item For every $q \in \mathcal{Q},$ there exists  $p\in \mathcal{P}$ such that $p,$ $q$ is a parallel pair and if $p_i,$ $q$ is a parallel pair then $q=q_i$ for $i=0,1.$
\item $\kappa_{\mathcal{P}}(p)<0$ for $p\neq p_0$, $\kappa_{\mathcal{Q}}(q_0)> 0,$ and $\kappa_{\mathcal{Q}}(q_1)\geq 0.$
\item $\mathcal{P}$, $\mathcal{Q}$ are curved in the same side at parallel pairs $p,$ $q$ close to $p_0,$ $q_0$ and $p_1,$ $q_1,$ respectively.
\end{enumerate}
Then, provided the curvatures of $\mathcal{P}$ and $\mathcal{Q}$ satisfy the following condition
\begin{align}\label{CondCurvSings}
\Big(\kappa_{\mathcal{Q}}(q_0)+\kappa_{\mathcal{P}}(p_0)\Big)\cdot
\Big(\kappa_{\mathcal{Q}}(q_1)+\kappa_{\mathcal{P}}(p_1)\Big)<0,
\end{align}
the Centre Symmetry Set of $\mathcal{P}\cup\mathcal{Q}$ has at least one asymptote.
\end{thm}

\begin{proof} We will adapt the method used in the proof of Proposition 3.3 in \cite{D-Z3}.

Let us first assume that $\kappa_{\mathcal{P}}(p_0)>0$.
Let $g:[t_0,t_1]\to\mathbb{R}^2,$ $f:[s_0,s_1]\to\mathbb{R}^2$ be the arc length parameterizations of $\mathcal{P}, \mathcal{Q}$, respectively. From (ii)-(iii) there exists a function $t:[s_0,s_1]\to[t_0,t_1]$ such that
\begin{align}\label{ParallelPropertyThm31}
f'(s)= - g'(t(s)).
\end{align}
By the implicit function theorem the function $t$ is smooth and $\displaystyle t'(s)=\frac{\kappa_{\mathcal{Q}}(f(s))}{\kappa_{\mathcal{P}}(g(t(s)))}$. 
From \eqref{CondCurvSings} we obtain that $\big(t'(s_0)+1\big)\cdot\big(t'(s_1)+1\big)<0$. Therefore by the Darboux theorem, there exists $s\in(s_0,s_1)$ such that $t'(s)=-1$. This means that $\kappa_{\mathcal{Q}}(f(s)) = - \kappa_{\mathcal{P}}(g(t(s))),$ which along with Remark \ref{rem:asymptote} yields the thesis.

If we assume that $\kappa_{\mathcal{P}}(p_0)=0$, then the proof is similar except that the domain of the function $t'$ is $(s_0,s_1]$. Therefore, we can then replace $t'(s_0)$ by the limit of $t'$ at $s_0$.
\end{proof}

In Figure \ref{FigWrinkled1} we illustrate two curves satisfying the assumptions of Theorem \ref{ThmCurvPosWrinkled}.

\begin{figure}[h]
\centering
\includegraphics[scale=0.6]{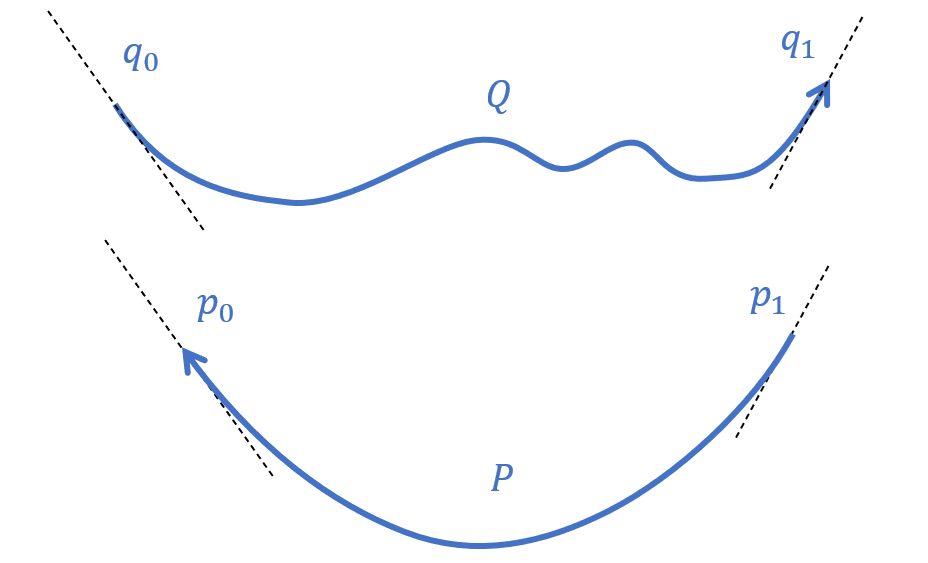}
\caption{Arcs satisfying assumptions of Theorem \ref{ThmCurvPosWrinkled}}
\label{FigWrinkled1}
\end{figure}

\begin{corollary}\label{CorThmCurvPosWrinkled}
Under the assumptions (i)-(iv) of Theorem \ref{ThmCurvPosWrinkled} and assuming that  $\kappa_{\mathcal{P}}(p_0)=\kappa_{\mathcal{Q}}(q_1)=0,$ the Centre Symmetry Set of $\mathcal{P}\cup\mathcal{Q}$ has at least one asymptote.
\end{corollary}

As a consequence of the Corollary \ref{CorThmCurvPosWrinkled}, we can make the following useful observation. 

\begin{corollary}
If there exist arcs curved in the same side $\overarc{p_k}{p_l}$ and $\overarc{p_{k'}}{p_{l'}}$ such that $p_k,$ $p_{l'}$ or $p_{k'},$ $p_l$ are inflexion points of $M,$ then $\CSS(M)$ has at least one asymptote.
\end{corollary}

\begin{definition}\label{RotationNumber}
A \textit{rotation number} of a regular curve is a rotation number of its velocity vector field.
\end{definition}

\begin{thm}\label{ThmAlgEvenNoOfCusp}
Let $C$ be a smooth branch of $\SC(M)$ which does not connect inflexion points. Let the maximal glueing scheme of $C$ be of the form
$$\begin{array}{ccccc}
a	&\frown&	 \ldots 	&\frown& a \\ \hline
b 	&\frown&	\ldots 	&\frown& b \\ \hline
\end{array},$$
where $a \neq b,$ and let $D$ be a smooth branch of $\CSS(M)$ produced by the maximal glueing scheme of $C.$ Then, the number of asymptotes of $D$ is equal to the number of cusps of $C.$ On the other hand, if the maximal glueing scheme of $C$ is of the form
$$\begin{array}{ccccc}
a	&\frown&	 \ldots 	&\frown& b \\ \hline
b 	&\frown&	\ldots 	&\frown& a \\ \hline
\end{array},$$ the number of asymptotes of $D$ is half the number of cusps of $C.$
\end{thm}
\begin{proof}
The secant caustic of $M$ is a union of smooth parameterized curves. Each of these curves will be called a \textit{smooth branch} of $\SC(M).$

The glueing scheme presented in Section \ref{decomp} produces not only smooth branches of $\CSS(M)$ but also smooth branches of the secant caustic $\SC(M),$ but in a slightly different way. 
Recall that the image of every maximal glueing scheme of $M$ under the $\CSS\text{-}$point map $\pi_{\CSS}$ is a smooth branch of the Centre Symmetry Set of $M.$
The \textit{secant map} of the curve $M$ can be defined as the following:
\begin{align*}
\pi_{\SC}: M\times M\to\mathbb{R}^2, (a, b)\mapsto a-b.
\end{align*}
Then, the image of every maximal glueing scheme of $M$ under the secant map $\pi_{\SC}$ is a smooth branch or a part of a smooth branch of the secant caustic of $M.$ 
For further details see \cite{Romero}.

We will now consider three possible versions of maximal glueing schemes whose images are smooth branches of $\CSS(M)$ and see in what way they translate to smooth branches of $\SC(M).$

If the maximal glueing scheme of $C$ is of the form
$\begin{array}{ccccccc}
a	&\frown&  \tilde{a}	&\frown&	 \ldots 	&\frown& a \\ \hline
b 	&\frown&  \tilde{b}	&\frown&	\ldots 	&\frown& b \\ \hline
\end{array},$ where $a \neq b,$ then the image of the maximal glueing scheme of $C$ under the secant map $\pi_{\SC}$ is a smooth branch. When using $\CSS$-point map, some of the obtained points correspond to the asymptotes of $\CSS(M).$ 
By Remark \ref{SCvsCSS} the same points, when using secant map, give rise to singular points of $\SC(M).$ Therefore in this case, the number of asymptotes of $\CSS(M)$ is equal to the number of cusps of $C.$

If the maximal glueing scheme of $C$ is of the form
$\begin{array}{ccccccc}
a	&\frown&  \tilde{a}	&\frown&	 \ldots 	&\frown& b \\ \hline
b 	&\frown&  \tilde{b}	&\frown&	\ldots 	&\frown& a \\ \hline
\end{array},$ where $a \neq b,$ then the image of the maximal glueing scheme of $C$ under the secant map $\pi_{\SC}$ is only a part of a smooth branch. Since $\pi_{\SC}(a,b) \neq \pi_{\SC}(b,a),$ for the image to be a smooth branch, the maximal glueing scheme needs to be prolonged into the form $\begin{array}{ccccccccccccc}
a	&\frown&  \tilde{a}	&\frown&	 \ldots 	&\frown& b	&\frown&  \tilde{b}	&\frown&	 \ldots 	&\frown& a \\ \hline
b 	&\frown&  \tilde{b}	&\frown&	\ldots 	&\frown& a	&\frown&  \tilde{a}	&\frown&	 \ldots 	&\frown& b \\ \hline
\end{array}.$
In such case, for each asymptote of $\CSS(M)$ there are always two points that (using secant map) give rise to different singular points of $\SC(M),$ which both correspond to this asymptote.
Therefore, the number of asymptotes of $\CSS(M)$ is two times smaller than the number of cusps of $C.$

Similarly to the previous case, if the maximal glueing scheme of $C$ is of the form
$\begin{array}{ccccccc}
a	&\frown&  \tilde{a}	&\frown&	 \ldots 	&\frown& b \\ \hline
a 	&\frown&  \tilde{b}	&\frown&	\ldots 	&\frown& b \\ \hline
\end{array},$ where $a \neq b$ and $a$ and $b$ are inflexion points, then again the image of the maximal glueing scheme of $C$ under the secant map $\pi_{\SC}$ is only a part of a smooth branch. For it to be a smooth branch, the maximal glueing scheme needs to be prolonged into the form $\begin{array}{ccccccccccccc}
a	&\frown&  \tilde{a}	&\frown&	 \ldots 	&\frown& b	&\frown&  \tilde{b}	&\frown&	 \ldots 	&\frown& a \\ \hline
a 	&\frown&  \tilde{b}	&\frown&	\ldots 	&\frown& b	&\frown&  \tilde{a}	&\frown&	 \ldots 	&\frown& a \\ \hline
\end{array}.$
In such case the number of asymptotes of $\CSS(M)$ is two times smaller than the number of cusps of $C.$
\end{proof}

In what follows, $T_{q}\mathcal{C}$ will denote the tangent space, which in our case, is a line tangent to the regular curve $\mathcal{C}$ at the point $q.$

\begin{thm}\label{pqTheorem2}
Let $\mathcal{P}$ and $\mathcal{Q}$ be regular curves with endpoints $p_0,$ $p_1$ and $q_0,$ $q_1,$ respectively. Let $\ell_0$ be a line through $q_0$ parallel to $T_{q_1}\mathcal{Q},$ $\ell_q$ be a line through $q_1$ parallel to $T_{q_0}\mathcal{Q}$ and $\ell_p'$ be a line through $\tau_{q_0-p_0}(p_1)$ parallel to $T_{q_0}\mathcal{Q}.$ Denote  $c=\ell_p'\cap T_{q_1}\mathcal{Q},$ $b_0=\ell_0\cap \ell_p'$ and $b_1=T_{q_0}\mathcal{Q}\cap T_{q_1}\mathcal{Q}$ and let us make the following assumptions.
\begin{enumerate}[(i)]
\item $T_{p_i}\mathcal{P}\| T_{q_i}\mathcal{Q}$ for $i=0,1.$
\item The curvature of $\mathcal{P}$ is positive and the curvature of $\mathcal{Q}$ is negative.
\item The absolute values of rotation number of $\mathcal{P}$ and $\mathcal{Q}$ are the same and smaller than $\frac{1}{2}.$
\item For every point $p \in \mathcal{P},$ there exists a point $q \in \mathcal{Q}$ and for every point $q \in \mathcal{Q}$ there exists a point $p \in \mathcal{P},$ such that $p,$ $q$ is a parallel pair.
\item $\mathcal{P}$ and $\mathcal{Q}$ are curved in the same side at every parallel pair $p,$ $q$ such that $p\in\mathcal{P}$ and $q\in\mathcal{Q}$.
\end{enumerate}

\noindent Let $\rho_{\max}$ (respectively $\rho_{\min}$) be the maximum (respectively minimum) of the set
\begin{align*}
\left\{\frac{c-b_1}{q_1-b_1}, \frac{c-b_0}{\tau_{q_0-p_0}(p_1)-b_0}\right\}.
\end{align*}

\noindent Then provided $\rho_{\max}<1$ or $\rho_{\min}>1$, the Centre Symmetry Set of $\mathcal{P}\cup\mathcal{Q}$ has at least one asymptote.
\end{thm}

\begin{figure}[h]
\centering
\includegraphics[scale=0.5]{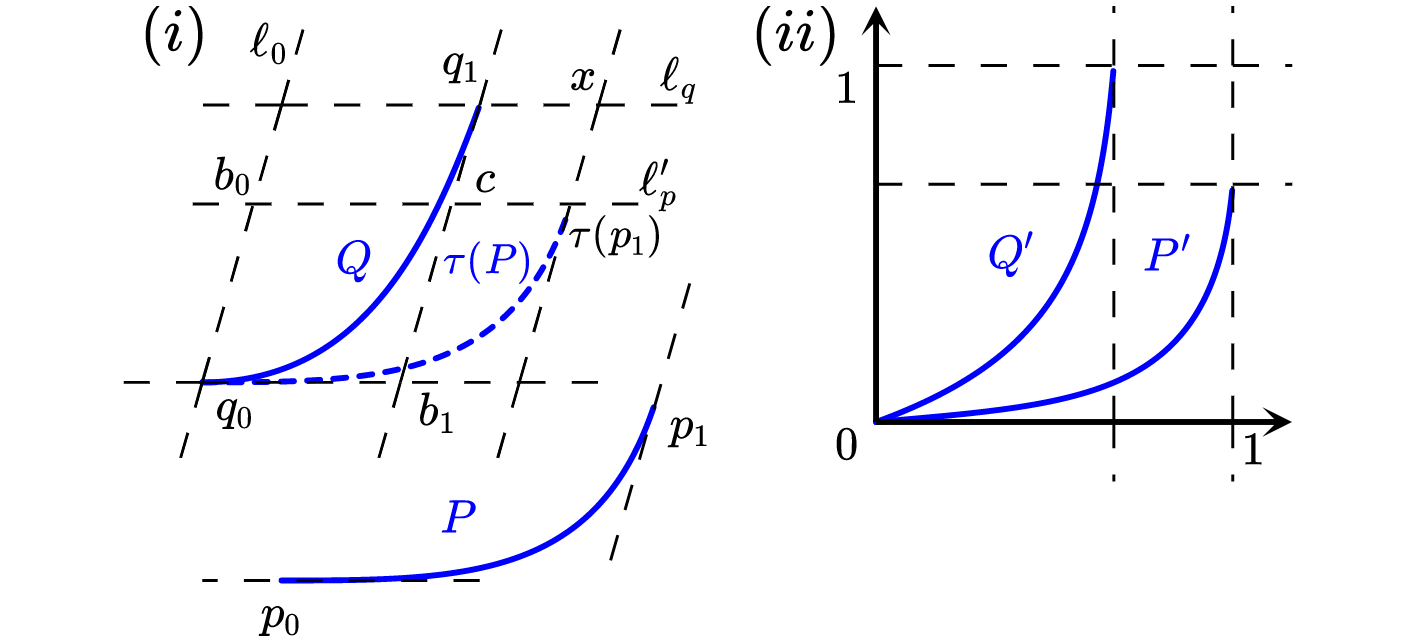}
\caption{Arcs satisfying the assumptions of Theorem \ref{pqTheorem2}}
\label{PictureThmParallelogram}
\end{figure}

Theorem \ref{pqTheorem2} can be proven by a straightforward adaptation of the method used in the proof of Theorem 2.26 in \cite{Romero} and of Proposition 3.7 in \cite{D-Z3}. The sketch of the proof is as follows: since a parallel pair that produces an asymptote of the $\CSS$ is an affine invariant, using an affine map we transform the arcs in Figure \ref{PictureThmParallelogram}(i) into Figure \ref{PictureThmParallelogram}(ii) such that the parallelogram with the endpoints at $q_0$ and $x$ goes to the unit square. Then, assuming there is no asymptote of the $\CSS$ we can find a contradiction.

In other words, we know more about the geometry of the Centre Symmetry Set of the sum $\mathcal{P}\cup\mathcal{Q},$ when the above conditions are fulfilled. 
These conditions allow us to compare pairs of regular curves $\mathcal{P}$ and $\mathcal{Q}$ after  translating one of them, so that they have a common endpoint. 
Let $x$ denote the point of intersection of $\ell_q$ with a line parallel to $\ell_0$ passing through $\tau(p_1)$ (see Figure \ref{PictureThmParallelogram}(i)). 
Then, we construct a parallelogram with opposite vertices $q_0$ and $\ell_q,$ based on the lines tangent to the curves. 
We compare the appropriate segments in the parallelogram in order to find out if both curves $\mathcal{P}$ and $\mathcal{Q}$ are contained in it (compare Figures \ref{PictureThmParallelogram}(i) and \ref{PictureThmParallelogramWrong}).
When this is the case, the conditions described above are fulfilled, and we can conclude that the Centre Symmetry Set of $\mathcal{P}\cup\mathcal{Q}$ has at least one asymptote.

\begin{figure}[h]
\centering
\includegraphics[scale=0.45]{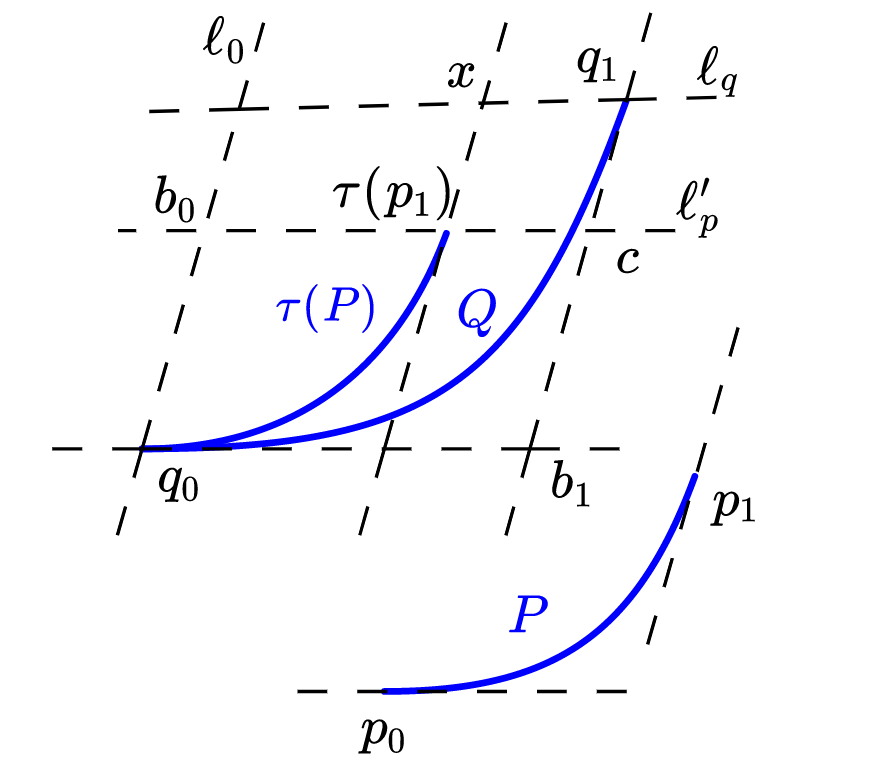}
\caption{Arcs not satisfying the assumptions of Theorem \ref{pqTheorem2}}
\label{PictureThmParallelogramWrong}
\end{figure}

\begin{prop}
Let $M$ a generic curve. If $\CSS(M)$ has an asymptote, then the Centre Symmetry Set approaches the asymptote from the opposite sides and the opposite ends as illustrated in Figure \ref{FigureAsymptotes}(i).
\end{prop}
\begin{proof}
\begin{figure}[h]
\centering
\includegraphics[scale=0.44]{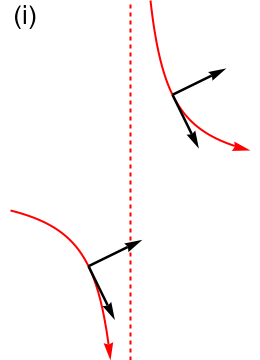}
\includegraphics[scale=0.44]{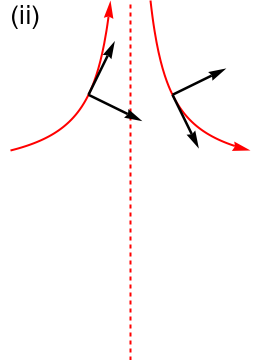}
\includegraphics[scale=0.44]{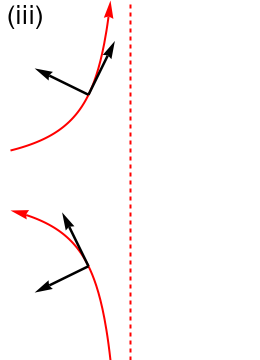}
\includegraphics[scale=0.44]{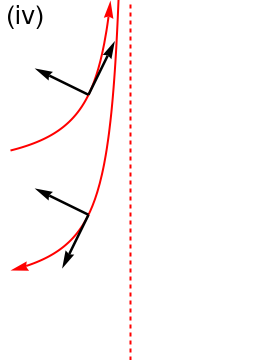}
\caption{An asymptote and branches approaching the asymptote}
\label{FigureAsymptotes}
\end{figure}
Let's analyse the four possibilities of approaching an asymptote by the Centre Symmetry Set in Figure \ref{FigureAsymptotes}. Note that by  genericity (see Theorem \ref{ThmGenericCSS}) there are no double asymptotes. In Figure \ref{FigureAsymptotes} we also illustrate the frames $(\mathbbm{t},\mathbbm{n})$, where $\mathbbm{t}$ is the unit tangent vector field, and $\mathbbm{n}$ is the unit continuous normal vector field to the $\CSS.$

Since the curvature of the $\CSS$ changes sign in a neighbourhood of an asymptote (see Proposition \ref{PropCurvCss}), the possibilities (ii) and (iii) in Figure \ref{FigureAsymptotes} cannot occur. Furthermore, the possibility (iv) is impossible because the orientation of the pair $(\mathbbm{t},\mathbbm{n})$ changes and it can only change in a neighbourhood of a singularity (see Theorem \ref{ThmGenericCSS}). This completes the proof. 
\end{proof}

\begin{definition}\label{RotationNumberCusps}
The \textit{rotation number} of a curve with at most cusp singularities is the rotation number of its continuous unit normal vector field.
\end{definition}

As Figure \ref{PictureNormalVectorToCusp} illustrates, a continuous unit normal vector field is well defined for cusp singularities. \begin{figure}[h]
\centering
\includegraphics[scale=0.35]{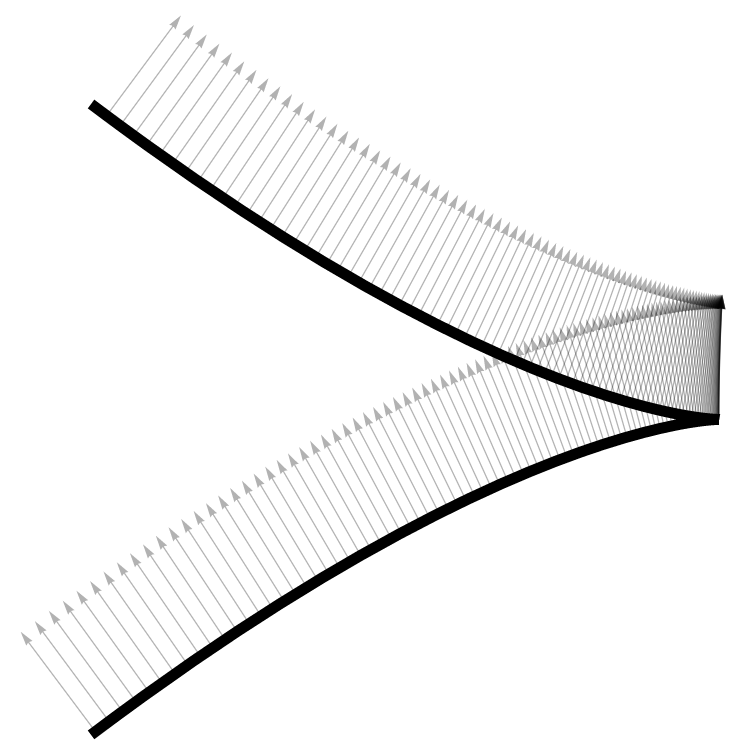}
\caption{A continuous normal vector field at a cusp singularity}
\label{PictureNormalVectorToCusp}
\end{figure}

Definition \ref{RotationNumberCusps} coincides with the classical Definition \ref{RotationNumber} of the rotation number for regular curves.

A number will be called \textit{half-integer} when it is equal to $\frac{s}{2},$ where $s$ is an odd integer.

\begin{lemma}\label{PropCuspsEq}
Let $C$ be a closed smooth curve with at most cusp singularities. If the rotation number of $C$ is an integer, then the number of cusps of $C$ is even and if the rotation number of $C$ is a half-integer, then the number of cusps of $C$ is odd.
\end{lemma}

\begin{proof}
A cusp singularity parameterized by $\beta$ consists of two regular open arcs. The pair $(\beta', \mathbbm{n}_{\beta}),$ where $\mathbbm{n}_{\beta}$ is a unit continuous normal vector field to $\beta$, changes orientation going through the cusp (see Figure  \ref{FigCoorCusp}). 

\begin{figure}[h]
\centering
\includegraphics[scale=0.35]{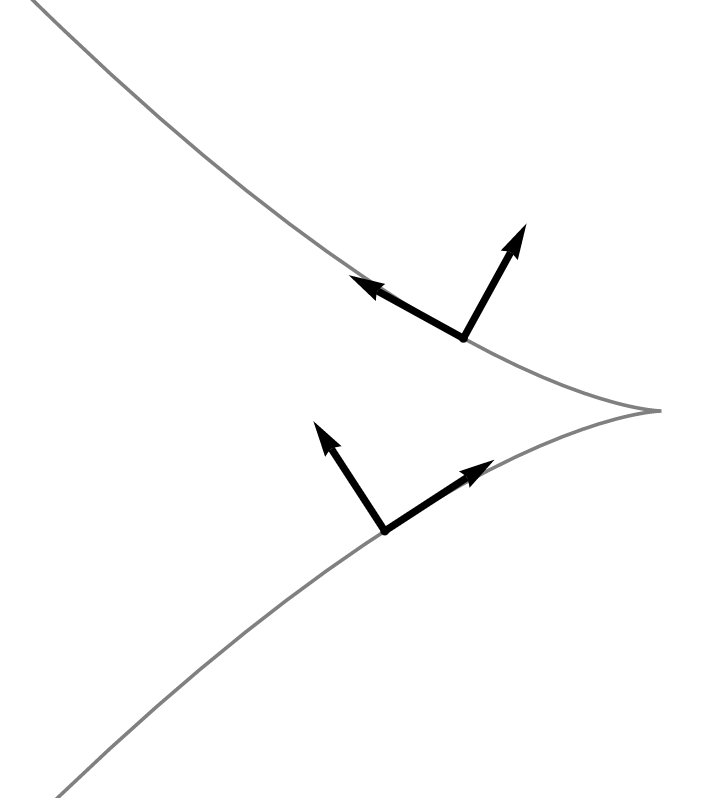}
\caption{Change of the orientation at a cusp singularity}
\label{FigCoorCusp}
\end{figure}

Let $\gamma_C:[0,2\pi]\to\mathbb{R}^2$ be a smooth parameterization of $C,$ where the set $[0,2\pi]$ modulo $2\pi$ is identified with $S^1.$ Let us fix a point $s_0$ such that $\gamma_C(s_0)$ is not a cusp singularity. The orientation of the pair $(\gamma'_C,\mathbbm{n}_{\gamma})$ at $s_0$ and $s_0+2\pi$ can be the same or opposite. If it is the same, then the rotation number of $C$ is an integer and the orientation of $(\gamma'_C,\mathbbm{n}_{\gamma})$ must have changed even number of times. In the other case, the rotation number of $C$ must be a half-integer and the orientation of $(\gamma'_C,\mathbbm{n}_{\gamma})$ must have changed odd number of times. 
\end{proof}

\begin{prop}\label{PropAlgParityOfCuspsInBranch}
Let $M$ be a generic regular closed curve. Let $\mathbbm{n}_{M}$ be a unit continuous normal vector field to $M$. Let $C$ be a smooth branch of $\CSS(M)$ which does not connect inflexion points. Then if the maximal glueing scheme of $C$ is of the form
$\begin{array}{ccccc}
\p_k	&\frown&	 \ldots 	&\frown& \p_l \\ \hline
\p_l 	&\frown&	\ldots 	&\frown& \p_k \\ \hline
\end{array},$ the number of cusps of $C$ is odd. Whereas if the maximal glueing scheme of $C$ is of the form
$\begin{array}{ccccc}
\p_k	&\frown&	 \ldots 	&\frown& \p_k \\ \hline
\p_l 	&\frown&	\ldots 	&\frown& \p_l \\ \hline
\end{array},$ the number of cusps of $C$ is even.
\end{prop}

\begin{proof}
When the maximal glueing scheme of $C$ is of the form
$\begin{array}{ccccc}
\p_k	&\frown&	 \ldots 	&\frown& \p_l \\ \hline
\p_l 	&\frown&	\ldots 	&\frown& \p_k \\ \hline
\end{array},$ the normal vectors to $M$ at $\p_{k}$ and $\p_{l}$ are opposite and the rotation number of $C$ is a half-integer. Using Lemma \ref{PropCuspsEq} we obtain that the number of cusps in $C$ is odd.

On the other hand, when the maximal glueing scheme of $C$ is in the form \linebreak
$\begin{array}{ccccc}
\p_k	&\frown&	 \ldots 	&\frown& \p_k \\ \hline
\p_l 	&\frown&	\ldots 	&\frown& \p_l \\ \hline
\end{array},$ the normal vectors to $M$ at $\p_{k}$ and $\p_{l}$ are the same and the rotation number of $C$ is an integer. By Lemma \ref{PropCuspsEq} the number of cusps in $C$ is even.
\end{proof}

\begin{thm}\label{ThmEvenNumberOnShell}
Let $M$ be a generic regular closed curve. Let $S^1\ni s\mapsto f(s)\in\mathbb{R}^2$ be a parameterization of $M$ and let $C$ be a smooth branch of the Centre Symmetry Set that connects two inflexion points $f(t_1)$ and $f(t_2)$ of $M.$ Then the number of inflexion points of the arc $f\big((t_1,t_2)\big)$ is even.
\end{thm}

\begin{proof}
If $M$ is a generic regular closed curve then the Wigner caustic $\Eq_{0.5}(M)$ is a union of smooth parameterized curves (e.g. see Theorem 2.13 in \cite{ZD_Wigner}). Each of these curves we will call a \textit{smooth branch} of the Wigner caustic of $M$.
The $0.5$-\textit{point map} is the map
\begin{align*}
\pi_{0.5}:M \times M\to\mathbb{R}^2, \quad  (a,b)\mapsto\dfrac{a+b}{2}.
\end{align*}
For further details about these terms and Wigner caustic see \cite{ZD_Wigner}.

The glueing scheme presented in Section \ref{decomp} produces not only smooth branches of $\CSS(M)$ (under the $\CSS\text{-}$point map $\pi_{\CSS}$) but also smooth branches of the Wigner caustic (under the $0.5\text{-}$point map $\pi_{0.5}$).

Therefore, smooth branches of $\CSS(M)$ and Wigner caustic must begin and end with the same points. From the construction of glueing scheme it follows that if a smooth branch of $\CSS(M)$ begins with an inflexion point $p,$ then it ends with a different inflexion point $q$ (see Theorem \ref{ThmGlueSchemeIsBranch} along with Lemma \ref{LemPropMaxGlueSchemes}).
This means that the same must be true for the Wigner caustic.

Theorem 4.10 in \cite{ZD_Wigner} shows that for a smooth branch of a Wigner caustic that connects two inflexion points $f(t_1)$ and $f(t_2)$ of $M,$ the number of inflexion points of the arc $f\big((t_1,t_2)\big)$ is even.
Therefore, the same must be true for a smooth branch of the Centre Symmetry Set connecting two inflexion points $f(t_1)$ and $f(t_2)$ of $M.$
\end{proof}

\begin{thm}\label{ThmCSSRosette}
Let $C_{n}$ be a generic $n$-rosette. Then
\begin{enumerate}[(i)]
\item the number of smooth branches of $\CSS(C_n)$ is equal to $n,$
\item exactly one smooth branch of $\CSS(C_n)$ has an odd number of cusps,
\item $n-1$ smooth branches of $\CSS(C_n)$ have a rotation number that is an integer and one branch has a rotation number that is a half-integer,
\item exactly $\displaystyle\left\lfloor\frac{n}{2}\right\rfloor$ smooth branches of $\CSS(C_n)$ have asymptotes,
\item (Proposition 5.13 in \cite{GB}) the number of cusp singularities of a smooth branch of $\CSS(C_n)$ which does not have an asymptote is not smaller than the number of cusp singularities of a smooth branch of the Wigner caustic of $C_n$ having the same glueing scheme as the smooth branch of $\CSS(C_n).$
\end{enumerate}
\end{thm}

\begin{proof}
Since $C_n$ is an $n$-rosette, for any point $a$ in $C_n$ there are exactly $2n-1$ points $b\neq a$ such that $a,$ $b$ is a parallel pair of $C_n.$ Therefore the set of parallel arcs is of the following form
$$\Phi_0=\left\{\overarc{\p_0}{\p_1}, \overarc{\p_1}{\p_2}, \ldots, \overarc{\p_{2n-2}}{\p_{2n-1}}, \overarc{\p_{2n-1}}{\p_0}\right\}.$$
Let $\CSS_{k}(C_n)$ be a smooth branch of $\CSS(C_n)$ and let $\Eq_{0.5,k}(C_n)$ be a smooth branch of $\Eq_{0.5}(C_n).$ We can create the following maximal glueing schemes.
\begin{itemize}
\item A maximal glueing scheme of $\CSS_{k}(C_n)$ and of $\Eq_{0.5,k}(C_n)$ for \linebreak $k\in\{1,2,\ldots,n-1\}:$
$$\begin{array}{ccccccccccccc}
\p_0	&\frown&	\p_1	&\frown&	\p_2	&\frown&	\ldots 	&\frown&	\p_{2n-2}	&\frown&	\p_{2n-1} 	&\frown&	\p_0 	\\ \hline 
\p_k	&\frown&	\p_{k+1}	&\frown&	\p_{k+2}	&\frown&	\ldots 	&\frown&	\p_{k-2}	&\frown&	\p_{k-1}	&\frown&	\p_k	\\ \hline
\end{array}.$$
\item A maximal glueing scheme of $\CSS_{n}(C_n)$ and of $\Eq_{0.5,n}(C_n):$
$$\begin{array}{ccccccccccc} 
\p_0	&\frown&	\p_1	&\frown&	\p_2	&\frown&	\ldots 	&\frown&	  \p_{n-1} 	&\frown& \p_n \\ \hline
\p_n	&\frown&	\p_{n+1} &\frown&	\p_{n+2} &\frown&	\ldots 	&\frown&	\p_{2n-1} 	&\frown& \p_0 \\ \hline
\end{array}.$$
\end{itemize}
The number of arcs of the glueing schemes presented above is $n(2n-1).$ There are no more maximal glueing schemes for the $\CSS(C_n),$ because Proposition \ref{PropNumDiffArcs} states that this is the total number of different arcs of the Centre Symmetry Set of $C_n$. Therefore the number of smooth branches of $\CSS(C_n)$ is equal to $n,$ which proves (i).

After applying Proposition \ref{PropAlgParityOfCuspsInBranch} to these maximal glueing schemes of the Centre Symmetry Set we obtain that there are $n-1$ smooth branches of $\CSS(C_n)$ with even number of cusps and one smooth branch of $\CSS(C_n)$ with odd number of cusps. This observation proves (ii).

Statement (ii) along with Lemma \ref{PropCuspsEq} imply (iii).

Let $(a_0, a_1, \ldots, a_{2n-1})$ be a sequence of points in $C_n$ with the order compatible with the orientation of $C_n$ such that $a_i,$ $a_j$ is a parallel pair. Then $C_n$ is curved in the same side at $a_i$ and $a_j$ if and only if $i-j$ is even.
Hence, smooth branches $\CSS_{2}(C_n), \CSS_{4}(C_n), \ldots, \CSS_{2\cdot\left\lfloor\frac{n}{2}\right\rfloor}(C_n)$ are created from parallel pairs $a,$ $b$ in $C_n$ such that $C_n$ is curved in the same side at $a$ and $b.$ All the other smooth branches of the Centre Symmetry Set of $C_n$ are created from parallel pairs $a,$ $b$ in $C_n$ such that $C_n$ is curved in different sides at $a$ and $b.$

By Lemma \ref{CurvedVsAsymptotes}, this means that for odd values of $k$ smooth branches $\CSS_{k}(C_n)$ cannot have any asymptotes.
Lemma 4.2 (iv) in \cite{Romero} states that the branches of the secant caustic which are created from the glueing schemes for the smooth branches of $\CSS_k(C_n)$ for even values of $k$ are singular. By Remark \ref{SCvsCSS} the smooth branch $\CSS_{k}(C_n)$ has at least one asymptote whenever $k$ is even. This proves (iv).

Let $a,$ $b$ be a parallel pair of $C_n$ and let $k$ be even. Let us consider branches $\CSS_k(C_n)$ and $\Eq_{0.5,k}(C_n)$. Note that $\CSS_k(C_n)$ does not have asymptotes. By Lemma \ref{CuspsofCSS} we know that $\CSS(M)$ has a singular point if and only if $\left( \frac{\kappa_M(a)}{\kappa_M(b)} \right)' = 0,$ where $'$ signifies the derivative with respect to the arc parameter of the curve.
Note that the Wigner caustic has a singular point if and only if $\frac{\kappa_M(a)}{\kappa_M(b)}  = 1$ (Proposition 2.8 in \cite{ZD_Wigner}) if $a,b$ is a standard pair of points.
Therefore, by Rolle's theorem, between each two cusps of the Wigner caustic we have at least one cusp singularity of $\CSS_k(C_n).$
As a result, we get that the number of cusp singularities of $\CSS_k(C_n)$ is not smaller than the number of cusp singularities of $\Eq_{0.5,k}(C_n).$
\end{proof}

\begin{remark}
Let $C_n$ be a generic $n$-rosette. By Theorem 5.3 in \cite{ZD_Wigner} we know that for even values of $n,$ every smooth branch of $\Eq_{0.5}(C_n)$ has an even number of cusps -- otherwise exactly one branch of $\Eq_{0.5}(C_n)$ has an odd number of cusps. By this observation and by Theorem \ref{ThmCSSRosette} we are able to find an example of a curve that has the number of cusps of the Centre Symmetry Set and the number of cusps of the Wigner caustic with different parities. Note that the parities of the numbers of these cusps in all the examples in other works have been the same (see \cite{DomitrzRios, GiblinHoltom}). In Figure \ref{FigCssWc} we present a $2$-rosette (the same as in Figure \ref{FigCssSc}) for which the number of cusps of the Wigner caustic is even, and the number of cusps of the Centre Symmetry Set is odd.
\end{remark}

\newpage

\begin{figure}[h]
\centering
\includegraphics[scale=0.30]{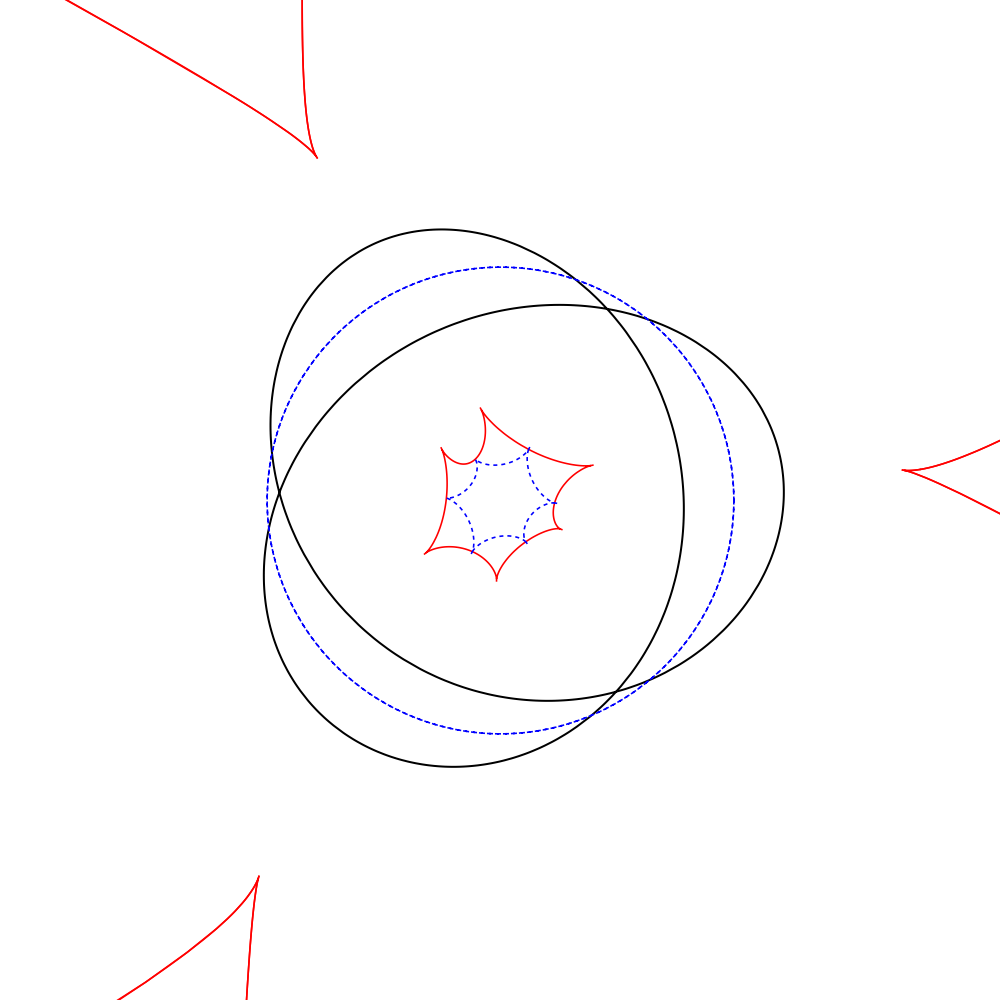}
\caption{A $2$-rosette $M$, $\CSS(M),$ and $\Eq_{0.5}(M)$ (the dashed line)}
\label{FigCssWc}
\end{figure}

\section*{Declarations}

\textbf{Funding and Competing interests}. The authors did not receive support from any organization for the submitted work. The authors have no competing interests to declare that are relevant to the content of this article.

\bibliographystyle{amsalpha}

\begin{thebibliography}{AAAA}

\bibitem{medial2} H.I. Choi, S.W. Choi, H.P. Moon, \emph{Mathematical theory of medial axis transform}, Pacific J. Math. 181 (1997), no. 1, 57--88.

\bibitem{Craizer0} M. Craizer, \emph{Iterates of Involutes of Constant Width Curves in the Mikowski Plane}, Beitrage zur Algebra und Geometrie, 55(2), 479--496, 2014.

\bibitem {CraizerDR1} M. Craizer, W. Domitrz, P.D.M. Rios, \emph{Even dimensional improper affine spheres}, J. Math. Anal. Appl. 421, 1803--1826 (2015).

\bibitem {CraizerDR2} M. Craizer, W. Domitrz, P.D.M. Rios, \emph{Singular improper affine spheres from a given Lagrangian submanifold}, Adv. Math. 374, 107326 (2020).

\bibitem {CraizerMartini} M. Craizer, H. Martini, \emph{Involutes of polygons of constant width in Minkowski planes}, Ars Math. Contemp. 11 (2016), no. 1, 107--125.

\bibitem{Craizer} M. Craizer, R.C. Teixeira, M.A.H.B. da Silva, \emph{Area distance of convex planar curves and improper affine spheres}, SIAM J. Imaging Sci. 1 (2008) 209--227.

\bibitem{CraizerPPOS} M. Craizer, R.C. Teixeira, M.A.H.B. da Silva, \emph{Polygons with Parallel Opposite Sides}, Discrete Comput Geom (2013) 50:474--490.

\bibitem{Art} I. Danielewska, D. Po\l awski, D. Sterczewska, M. Zwierzy\'nski, \emph{Arthistic Aspects of the Wigner Caustic and the Centre Symmetry Set}, arXiv:2409.04443

\bibitem{DJRR} W. Domitrz, S. Janeczko, P.D.M. Rios, M.A.S. Ruas, \emph{Singularities of Affine Equidistants: extrinsic
geometry of surfaces in 4-space}, Bull. Braz. Math. Soc. 47(4), 1155--1179 (2016).

\bibitem{DomitrzMR} W. Domitrz, M. Manoel, P. de M. Rios, \emph{The Wigner caustic on shell and singularities of odd functions}, Journal of Geometry and Physics 71(2013), pp. 58--72.

\bibitem{Romero} W. Domitrz, M.C. Romero Fuster, M. Zwierzy\'ski, \emph{The geometry of the secant caustic of a planar curve}, Differ. J. Appl. 78, 101797 (2021).

\bibitem{DomitrzRios} W. Domitrz, P. de M. Rios, \emph{Singularities of equidistants and global centre symmetry sets of Lagrangian submanifolds}, Geom. Dedicata 169 (2014), 361--382.

\bibitem {DomitrzRiosRuas} W. Domitrz, P.D.M. Rios, M.A.S. Ruas, \emph{Singularities of affine equidistants: projections and contacts}, J. Singul. 10, 67--81 (2014).

\bibitem{D-Z3} W. Domitrz, M. Zwierzy\'nski, \emph{Singular Points of the Wigner Caustic and Affine Equidistants of Planar Curves}, Bull Braz Math Soc, New Series 51, 11--26 (2020).

\bibitem{GB} W. Domitrz, M. Zwierzy\'nski, \emph{The Gauss-Bonnet theorem for coherent tangent bundles over surfaces with boundary and its applications}, J. Geom. Anal. 30, 3243--3274 (2020).

\bibitem{ZD_Wigner} W. Domitrz, M. Zwierzy\'nski,\emph{The Geometry of the Wigner Caustic and a Decomposition of a Curve into Parallel Arcs}, Analysis and Mathematical Physics (2022) 12, 7.

\bibitem {GiblinHoltom} P.J. Giblin, P.A. Holtom, \emph{The centre symmetry set}, Geometry and Topology of Caustics, vol. 50, pp. 91--105. Banach Center Publications, Warsaw (1999).

\bibitem{GiblinReeve} P.J. Giblin, G. Reeve, \emph{Centre Symmetry Set of Families of Plane Curves}, Demonstratio Mathematica, Vol. XLVIII No 2 (2015).

\bibitem{GiblinReeve2} P.J. Giblin, G. Reeve, \emph{Equidistants for families of surfaces}, J. Singul. 21 (2020), 97--118.

\bibitem {GiblinZ1} P.J. Giblin, V.M. Zakalyukin, \emph{Singularities of systems of chords}, Funct. Anal. Appl. 36, 220--224 (2002).

\bibitem {GiblinZ2} P.J. Giblin, V.M. Zakalyukin, \emph{Singularities of centre symmetry sets}, Proc. London. Math. Soc 90, 132--166 (2005).

\bibitem {GiblinZ3} P.J. Giblin, V.M. Zakalyukin, \emph{Recognition of centre symmetry set singularities}, Geom. Dedicata 130,
43--58 (2007).

\bibitem {GG-Book} M. Golubitsky, V. Guillemin, \emph{Stable Mappings and Their Singularities}, Springer, Berlin (1974).

\bibitem {Janeczko}  S. Janeczko, \emph{Bifurcations of the center of symmetry}, Geom. Dedicata 60 (1996), 9--16.

\bibitem{medial1} S. Pizer, K. Siddiqi (eds), \emph{Medial representations, Computational Imaging and Vision}, vol. 37 (Springer, 2008).

\bibitem{ReeveZ} G.M. Reeve, V.M. Zakalyukin, \emph{Singularities of the Affine Chord Envelope for Two Surfaces in Four-Space}, Proceedings of the Steklov Institute of Mathematics, 2012, Vol. 277, pp. 221--233.

\bibitem {UYBook} M. Umehara, K. Yamada, \emph{Differential Geometry of Curves and Surfaces}, World Scientific Publishing, 2017.

\bibitem {Zhang1} D. Zhang, \emph{The lower bounds of the mixed isoperimetric deficit}, Bull. Malays. Math. Sci. Soc. 44, 2863--2872 (2021).

\bibitem{Zwierz1} M. Zwierzy\'nski, \emph{The improved isoperimetric inequality and the Wigner caustic of planar ovals}, J. Math. Anal. Appl 442(2), 726--739 (2016).

\bibitem{Zwierz2} M. Zwierzy\'nski, \emph{The Constant Width Measure Set, the Spherical Measure Set and Isoperimetric Equalities for Planar Ovals}, arXiv:1605.02930

\bibitem{Zwierz3} M. Zwierzy\'nski, \emph{Isoperimetric equalities for rosettes}, Int. J. Math. 31(5), 2050041 (2020).

\end{thebibliography}

\end{document}